\documentclass[reqno]{amsproc}
\usepackage[utf8]{inputenc}

\usepackage{pst-all}
\usepackage{amsthm}
\usepackage{mathtools}
\usepackage{diagbox}
\usepackage{dsfont}
\usepackage{multirow}

\usepackage{amsfonts}
\usepackage{graphicx}
\usepackage{amscd}
\usepackage{amsmath}
\usepackage{amssymb}
\usepackage{latexsym}
\usepackage{amsmath,amssymb,latexsym}
\usepackage{hyperref}
\usepackage[all]{xy}
\usepackage{color}
\usepackage{mathrsfs}
\theoremstyle{plain}

\newtheorem{corol}{\bf Corollary}

\newtheorem{example}{\bf Example}
\newtheorem{conj}{\bf Conjecture}
\newtheorem{observation}{\bf Observation}

\newtheorem{prop}{\bf Proposition}

\newtheorem{teor}{\bf Theorem}
\numberwithin{equation}{section}

\begin{document}
\title[Symmetries of Ricci Flows]{Symmetries of Ricci Flows}
\author[Enrique López]{Enrique  López$^1$}
\author[Stylianos Dimas]{Stylianos Dimas$^2$}
\author[Yuri Bozhkov]{Yuri Bozhkov$^3$}
\address{$^1$Institute of Mathematics, Statistics and Scientific Computing, University of Campinas - UNICAMP.}
\address{$^{2}$Department of Mathematics, Technological Institute of Aeronautics - ITA.}
\address{$^3$Institute of Mathematics, Statistics and Scientific Computing, University of Campinas - UNICAMP.}
\email{$^1$elopez@ime.unicamp.br}
\email{$^2$stylianos.dimas.ita@gmail.com}
\email{$^3$bozhkov@unicamp.br}
%
\urladdr{$^1$https://www.ime.unicamp.br}
\urladdr{$^2$https://www.ita.br}
\urladdr{$^3$https://www.ime.unicamp.br}

\begin{abstract}
In the present work, we find the Lie point symmetries of  the Ricci flow on an $n$-dimensional manifold, and we introduce a method in order to reutilize these symmetries to obtain the Lie point symmetries of particular metrics. We apply this method to retrieve the Lie point symmetries of the Einstein equations --- seen as a ``static'' Ricci flow ---, and of some particular types of metrics of interest, such as,  on warped products of manifolds. Finally, we use the symmetries found to obtain  invariant solutions of the Ricci flow for the particular families of metrics considered.
\end{abstract}

\maketitle
\section{Introduction}
Hamilton, in 1982~\cite{hamilton1982three}, introduced the Ricci flow: A nonlinear geometric evolution equation in which one starting with a smooth Riemannian manifold $\left(M^{n}, g_{0}\right)$  deforms its metric using the following equation
\begin{equation*}\label{equa01}
\frac{\partial}{\partial t} g=-2 Ric_{g(t)},
\end{equation*}
where $Ric_{g(t)}$ denotes the Ricci tensor of the metric $g(t)$ and $g(0)=g_0$. Similar to the heat flow of temperature distributions and other diffusion processes, the Ricci flow deforms the geometry towards more uniform ones whose limit allows us to draw topological conclusions about $M^n$. For this reason, the Ricci flow has been studied and successfully applied to solve important manifold classification problems, such as the famous Poincaré's Conjecture and the Thurston's geometrization conjecture. Perelman completely proved these two conjectures in 2003~\cite{perelman2003finite}.  There is a huge number of works dedicated to the study of various aspects of the Ricci flow, amongst them, the study of the Black Holes Theory, see Xing and Xiang~\cite{xing2017mathematical}, and Bakas~\cite{bakas2004ricci}. 
\ 

It is well known that in contrast to the theory of linear differential equations --- although the great progress in the past seven decades ---   there is still no unitary theory for nonlinear differential equations and systems. Nevertheless, various  effective approaches have been developed: Compact operators and operators of monotone type, topological degree theories, fixed point theory and modern variational principles, critical point theory, Morse theory and critical groups, treated from a general qualitative point of view \cite{kr,prr}; Rellich and Pohozaev Identities via a Noeterian approach \cite{ye1,ye2,ye3}; analytical methods for constructing exact solutions of differential equations \cite{mel2}, etc.

\ 

We believe that the symmetry approach based on the
 application of Lie symmetry methods for differential equations may be 
 fruitful in the Nonlinear Analysis.  We note that Sophus Lie originally introduced in the last part of the $19$th century the general theory of Lie transformation groups specifically to study differential equations.  After a period of dormancy, Lie's ideas resurfaced, beginning with the pioneering works of L. Ovsyannikov and his collaborators. It is no exaggeration  to state that the great array of applications of Lie group methods to nonlinear differential equations has been one of the great successes of the last 60 years. There is by now a vast literature in this area, and many of the basic methods and applications can be found in \cite{olver2000applications,ole,ovsiannikov2014group} as well as the three volume handbook \cite{ibr4}.  Some of the directions being developed include the construction of explicit symmetric solutions through symmetry reduction, the connection between symmetries and conservation laws via the Noether's Theorem for problems admitting variational structures, construction and classification of differential invariants and, then, invariant differential equations and variational principles, separation of variables and integrability of both linear and nonlinear partial differential equations, asymptotic behavior of solutions for long times and near blow-up using symmetric solutions, and the design of symmetry-preserving numerical algorithms, just to mention a few.
 
\ 
 
We recall that the symmetries of a differential equation transform
solutions of the equation to other solutions. One of the main
benefits of this theory is that by following a completely algorithmic procedure, one is able to determine the symmetries of a
differential equation or systems of differential equations. They comprise a structural property of the equation --  in essence they are equation's DNA. The knowledge of the symmetries of an equation enables one to use them
for a variety of purposes, from obtaining analytical solutions and reducing its order to finding of integrating factors and conservation
laws as we mentioned above. In fact, many, if not all, of the different empirical methods for solving ordinary differential equations (ODE) we have learned from standard courses at the undergraduate level emerge from a symmetry. For instance, having at our
disposal a Lie point symmetry of a first order ODE, we can immediately get explicitly an integrating factor by a formula obtained by S. Lie. In this regard, we remind the words of Nail Ibragimov that 
\begin{quote}``one of the most remarkable achievements of Lie was the discovery that the
majority of known methods of integration of ordinary differential equations,
which seemed up to that time artificial and internally disconnected, could be
derived in a unified manner using the theory of groups. Moreover, Lie provided a classification of all ordinary differential equations
of arbitrary order in terms of the symmetry groups they admit and thus
described the whole collection of equations for which integration or lowering
of the order could be effected by group-theoretical methods. However, these
and other very valuable results he obtained could not for a long time be
widely disseminated and remained known to only a few. It could be said that
this is the state of affairs today with methods of solution of problems of
mathematical physics: Many of these are of a group-theoretical nature, but
are presented as a result of a lucky guess'', \cite{i2}.\end{quote}

\ 

All in all, symmetries play a far-reaching role in the analysis of differential equations: They allow us to recover the full picture from partial information. For this reason, it is not an uncommon practice to employ symmetries for constructing new solutions from known ones or for reducing a differential equation. At this point, is worth mentioning that for the main bulk of calculations, we used the symbolic package SYM~\cite{dimas2004sym}.

\ 

So far, we only have a description of the Lie point symmetries of the Ricci flow for the two-dimensional case, see 
Cimpoiasu and Constantinescu~\cite{cimpoiasu2006symmetries} and Wang~\cite{wang2013symmetries}. The goal of the present paper is to carry out a complete group classification of the Ricci flow on $n$-dimensional manifolds and its use in constructing analytical solutions for particular metrics. Our central result is the following: 

\begin{teor}\label{teo1rfn}
The Lie algebra of the classical symmetries of the Ricci flow on a Riemannian manifold $(M^n,g)$ with $n\geq 2$ is spanned by:
\begin{eqnarray*}\label{simericciflow}
\begin{aligned}
    X_1=&\frac{\partial}{\partial t},\\
    X_2=&t\frac{\partial}{\partial t}+\sum_{i=1}^n \sum_{j\geq i}^n g_{ij}\frac{\partial}{\partial g_{ij}},\\
    X_{k+2}=&\xi^k\frac{\partial}{\partial x^k}-\sum_{i=1}^n \sum_{j\geq i}^n \left( g_{ki}\frac{\partial\xi^k}{\partial x^j}+g_{kj}\frac{\partial\xi^k}{\partial x^i}\right)\frac{\partial}{\partial g_{ij}},
\end{aligned}
\end{eqnarray*}
where $\xi^1, \cdots ,\xi^n$ are arbitrary smooth functions of $x^1,\cdots, x^n$, $k=1,\cdots,n$ and $g_{ij}$ are the coefficients of the metric tensor $g$.
\end{teor}  

This paper is organized as follows: in Section $2$, we give a brief introduction of all the notions utilized in the rest of the text as well as we illustrate the key method that we employ for obtaining the Lie point symmetries of the Ricci flow for particular families of metrics from the Lie algebra of Theorem~\ref{teo1rfn}. Our main result resides in Section $3$, where we determine the Lie point symmetries of the Ricci flow for the $n$-dimensional case as well as its optimal system for its finite--dimensional sub algebra. In the next two Sections, we show how this general result can be utilized to obtain the Lie point symmetries for particular metrics: We start the exhibition  by retrieving the Lie symmetries of the Einstein equations for the $n$-dimensional case. Next, we determine the Lie point symmetries of the Ricci flow for product manifolds expressed by warped and doubly-warped metrics. Then, we build invariant solutions from these symmetries.  Finally, Section $6$ contains comments and concluding remarks.

\section{Preliminaries}
In this section, we will briefly introduce the general concept of Lie point symmetry of differential equations. For a detailed approach, we recommend reading~\cite{hydon2000symmetry,ibr4,olver2000applications,ovsiannikov2014group}. We restrict our attention to connected local Lie groups of symmetries, leaving aside discrete symmetries.

Let assume a system of differential equations of $n$-th order with $p$ independent and $q$ dependent variables 
\begin{equation*}
    \Delta_\nu(x,u^{(n)})=0,\quad\quad\quad\nu=1,\dots,l,
\end{equation*}
involving the derivatives of $u$ with respect to $x$ up to order $n$. Moreover, we consider  all involved functions, vector fields, and tensors sufficiently smooth in their arguments. Note that the aforementioned system  can be viewed as a smooth map from the jet space $V\times U^{(n)}$ ($V$ is an open set of $\mathds{R}^p$ while $U^{(n)}$ is the Cartesian product space whose coordinates represent all  the partial derivatives, up to order $n$, of the components of the vector function $u$) to $\mathds{R}^l$:
\begin{equation*}
    \Delta:V\times U^{(n)}\rightarrow \mathds{R}^l.
\end{equation*}
If $0\in \mathds{R}^l$ is a regular value of the previous mapping, we have that 
\begin{equation*}
\mathscr{S}_{\Delta}=\left\{\left(x, u^{(n)}\right): \Delta\left(x, u^{(n)}\right)=0\right\} \subset V \times U^{(n)},
\end{equation*}
determines a sub-manifold of the jet space. Henceforth we will consider zero to be a regular value.

A smooth solution of the given system of differential equations is a smooth vector function $u=f(x)$ such that
\begin{equation*}
    \Delta_\nu (x,f^{(n)})=0\quad\quad\quad\nu=1,\dots,l.
\end{equation*}
This is just a restatement of the fact that the derivatives $\partial_Jf^\alpha(x)$ of $f$ must satisfy the algebraic constraints imposed by the system of differential equations.

A Lie point symmetry group of this system is a local group of transformations which maps solutions to solutions. That is, if $G$ be a local group of transformations acting on $V\times U$, then $g(x,u^{(n)})\in \mathscr{S}_{\Delta}$ whenever $(x,u^{(n)})\in\mathscr{S}_{\Delta}$ and $g\in G$.

The tangent space of a  local group of transformations is a Lie algebra --- its vectors are often called infinitesimal generators. With starting point a Lie algebra by employing its  exponential mapping, we can retrieve the  local group of transformations with each element of the Lie algebra giving rise to a local subgroup of transformations called (local) one-parameter group. Through this local isomorphism, it is common practice to identify a Lie point symmetry with an infinitesimal generator. 

\begin{teor} 
\cite[Theorem~2.31]{olver2000applications}. Suppose that
\begin{equation*}
\Delta_{\nu}\left(x, u^{(n)}\right)=0, \quad \nu=1, \ldots, l,
\end{equation*}
is a system of differential equations over $M \subset V \times U$ with zero being a regular value of $\Delta$. If $G$ is a local group of transformations acting on $M$, $\mathcal G$ its corresponding Lie algebra and
\begin{equation*}
\operatorname{pr}^{(n)} \mathbf{X}\left[\Delta_{\nu}\left(x, u^{(n)}\right)\right]=0, \quad \nu=1, \ldots, l, \text { whenever } \quad \Delta\left(x, u^{(n)}\right)=0,
\end{equation*}
for every infinitesimal generator $\mathrm{X}$ of $\mathcal G$, then $G$ is a Lie point symmetry group of the system.
\end{teor}

From these infinitesimal generators, it is possible to reduce the number of independent variables of the given system of equations and even obtain solutions. The solutions thus obtained are called invariant solutions of the infinitesimal generator employed. Thus, the natural question that arises once the set of symmetries has been calculated is how to obtain all the invariant solutions. This leads to the concept of optimal system, see~\cite[Proposition~3.6]{olver2000applications}.

The main objective of our study is the Lie point symmetries of the Ricci Flow for a general metric. A difficult and laborious task  indeed! However, from these symmetries --- by a method we developed --- we will be able to retrieve the symmetries for any particular family of metrics by restricting the generic Lie algebra to the metric at hand and hence eliminating the need to repeat each time the same copious process for finding the Lie point symmetries. This method is based on a particular form of the infinitesimal generator known as the canonical form, or simply, the characteristic.

Recall that given an infinitesimal generator
\begin{equation*}
    X=\sum_{i=1}^{p}\xi^i(x,u)\frac{\partial}{\partial x^i}+\sum_{j=1}^q\eta^j(x,u)\frac{\partial}{\partial u^j},
\end{equation*}
its canonical form is the vector
\begin{equation*}
  X=\sum_{j=1}^q Q_j\frac{\partial}{\partial u^j},
\end{equation*}
where
\begin{equation*}
    Q_j=\eta^j(x,u)-\sum_{i=1}^p\xi^i(x,u)\frac{\partial u^j}{\partial x^i}.
\end{equation*}
In the following example, we illustrate the method.
\begin{example}\label{ex:1}
We wish to obtain the Ricci flow symmetries of the metric
\begin{equation*}
    g(x^1,x^2,t)=e^{u(x^1,x^2,t)}\left(dx^1\otimes dx^1+dx^2\otimes dx^2\right),
\end{equation*}
from the symmetries of Theorem~\ref{teo1rfn}. Let 
\begin{equation*}
X=c_1X_1+c_2X_2+X_3+X_4+X_{5},
\end{equation*}
where $c_1,c_2\in\mathds{R}$, be a symmetry for the generic metric as provided by the Theorem~\ref{teo1rfn}. First, we write $X$ in its canonical form:
\begin{equation*}
Q_X=\sum_{1\leq i\leq j\leq 2}Q_{ij}\left(t,x,g_{ps},\frac{\partial g_{ps}}{\partial t},\frac{\partial g_{ps}}{\partial x^r}\right)\frac{\partial}{\partial g_{ij}},
\end{equation*}
where
\begin{equation*}\label{caracgeneral}
	\begin{split}
			Q_{ij}=&c_2 g_{ij}(x)-\sum_{1\leq s\leq 2}\left(g_{si}\frac{\partial\xi^s}{\partial x^j}+g_{sj}\frac{\partial\xi^s}{\partial x^i}\right)-\sum_{1\leq s\leq 2}\xi^s(x)\frac{\partial g_{ij}}{\partial x^s}\\
			&-(c_1+c_2t)\frac{\partial g_{ij}}{\partial t}.
	\end{split}
\end{equation*}
Now, looking at the metric, it is easy to see that we need to impose the following restrictions to the generic metric of dimension $2$: $S=\{g_{11}=g_{22}=e^u$ and $g_{12}=0\}$. In addition, for this particular metric, the infinitesimal generator in the canonical form will be like
\begin{equation*}
\left.Q_X\right\rvert_S=Q\left(x^1,x^2,t, u,\frac{\partial u}{\partial t},\frac{\partial u}{\partial x^1},\frac{\partial u}{\partial x^2}\right)\frac{\partial}{\partial u}.
\end{equation*}
We are ready to give the restrictions that we need to impose to the Lie algebra of Theorem~\ref{teo1rfn} for $n=2$:
\begin{multline*}
	\begin{cases}Q_X(g_{11}) &= \left.Q_X\right\rvert_S(e^u)\\ Q_X(g_{12}) &= \left.Q_X\right\rvert_S(0)\\ Q_X(g_{22}) &= \left.Q_X\right\rvert_S(e^u)\end{cases}\implies \begin{cases}Q_{11} &= e^u Q\\  Q_{12} &= 0\\ Q_{22} &= e^u Q\end{cases}\implies\\
	\begin{cases}
		-\left(\xi^1\frac{\partial u}{\partial x^1}+\xi^2\frac{\partial u}{\partial x^2}+(c_1+c_2t)\frac{\partial u}{\partial t}-c_2+2\frac{\partial \xi^1}{\partial x^1}\right)e^u &=  e^u Q\\
		\left(\frac{\partial\xi^1}{\partial x^2}+\frac{\partial\xi^2}{\partial x^1}\right)e^u &=0\\
		-\left(\xi^1\frac{\partial u}{\partial x^1}+\xi^2\frac{\partial u}{\partial x^2}+(c_1+c_2t)\frac{\partial u}{\partial t}-c_2+2\frac{\partial \xi^2}{\partial x^2}\right)e^u &=  e^u Q
	\end{cases}\implies\\
	\begin{cases}
	Q\left(x^1,x^2,t, u,\frac{\partial u}{\partial t},\frac{\partial u}{\partial x^1},\frac{\partial u}{\partial x^2}\right) &= -\left(\xi^1\frac{\partial u}{\partial x^1}+\xi^2\frac{\partial u}{\partial x^2}+(c_1+c_2t)\frac{\partial u}{\partial t}-c_2+2\frac{\partial \xi^1}{\partial x^1}\right)\\
	\frac{\partial \xi^1}{\partial x^1}-\frac{\partial \xi^2}{\partial x^2}&=0\\
	\frac{\partial\xi^1}{\partial x^2}+\frac{\partial\xi^2}{\partial x^1} &=0.
	\end{cases}
\end{multline*}
Therefore, the Lie algebra for the restricted problem is:
\begin{equation*}\label{eqtor1}
    X= (c_1+c_2 t)\frac{\partial}{\partial t}+\xi^1\frac{\partial}{\partial x^1}+\xi^2\frac{\partial}{\partial x^2}+\left(c_2-2\frac{\partial\xi^1}{\partial x^1}\right)\frac{\partial}{\partial u},
\end{equation*}
where $\xi^1$ and $\xi_2$ are arbitrary functions of $x^1$ and $x^2$ satisfying the Cauchy-Riemann equations. Note that in \cite{cimpoiasu2006symmetries} the authors obtained the very same Lie algebra from the ground up. 
\end{example}

\section{The Lie point symmetries of the Ricci Flow}
In this section, we compute the  Lie point symmetries of the Ricci Flow for the $n-$dimensional case and we classify its finite-dimensional sub algebra. We begin by stating and proving our main result.
 
By direct calculation, using SYM, of the Lie point symmetries for the cases $n=2,3\text{ and } 4$ we arrived to the following conjecture: 
 \begin{conj}
 The Lie algebra of the classical symmetries, Lie point symmetries, of the Ricci flow on a Riemannian manifold $(M^n,g)$ with $n\geq 2$ is spanned by:
\begin{eqnarray}\label{simericciflowConj}
\begin{aligned}
    X_1=&\frac{\partial}{\partial t},\\
    X_2=&t\frac{\partial}{\partial t}+\sum_{i=1}^n \sum_{j\geq i}^n g_{ij}\frac{\partial}{\partial g_{ij}},\\
    X_{k+2}=&\xi^k\frac{\partial}{\partial x^k}-\sum_{i=1}^n \sum_{j\geq i}^n \left( g_{ki}\frac{\partial\xi^k}{\partial x^j}+g_{kj}\frac{\partial\xi^k}{\partial x^i}\right)\frac{\partial}{\partial g_{ij}},
\end{aligned}
\end{eqnarray}
where $\xi^1, \cdots ,\xi^{n-1}$ e $\xi^n$ are arbitrary smooth functions of $x^1,\cdots, x^n$, $k\in\{1,\cdots,n\}$ and $g=(g_{ij})$ is the metric tensor.
 \end{conj}
 
 It is easy to verify that the above generators leave invariant the Ricci flow for  $n\ge2$:
 \begin{itemize}
 	\item The symmetry $X_1$ is obvious since the Ricci flow equations do not involve terms containing explicitly the variable $t$.
	\item As for $X_2$, observe that the Ricci tensor is invariant under scalings, that is $Ric_{cg}=Ric_{g}$ for all $c>0$. Hence, by the one-parameter group generated by $X_2$,
\begin{equation*}
	\Psi_\varepsilon(x,t,g_{ij}) = (x,e^{\varepsilon}t, e^{\varepsilon}g_{ij}),
\end{equation*}
we get that 
\begin{equation*}
	\hat{g}_{ij}(\hat x,\hat t)=e^{\varepsilon}g_{ij}(\hat x,e^{-\varepsilon}\hat t),
\end{equation*} 
and
\begin{equation*}
	\frac{\partial}{\partial \hat t} \hat{g}_{ij}= \frac{\partial}{\partial  t} g_{ij}(x, t)=-2Ric_{g_{ij}}=-2Ric_{e^{-\varepsilon}\hat g_{ij}}=-2Ric_{\hat{g}_{ij}}.
\end{equation*}
\item Finally, let us turn to the infinite ideal $X_{k+2}$. We will show that these symmetries arise from the covariance  of the Ricci flow equations. Let $\hat{x}=\Psi_{\varepsilon}(x)$ be the flow determined by the initial value problem: 
\begin{equation*}
\begin{cases}\frac{d\hat{x}^k}{d\varepsilon}\Big|_{\varepsilon=0} & =\xi^k(\hat x), \\ \hat{x}^k\Big|_{\varepsilon=0} & =x^k.\end{cases}
\end{equation*}
And let $g_{ij}$ be the metric in coordinates $x$. Since it is a tensor, it is covariant. So in any other coordinate system, $\hat x$, it has the form:
\begin{equation*}\label{anali01}
    \hat{g}_{sl}=g_{ij}\frac{\partial x^i}{\partial \hat{x}^s}\frac{\partial x^j}{\partial \hat{x}^l}.
\end{equation*}
We wish to analyse  $\hat{g}_{sl}$ around $\varepsilon=0$, hence obtaining the infinitesimal change of $\hat{g}_{sl}$. By the properties of the flow we know that $ x = \Psi_{-\varepsilon}(\hat x)$. Thus
\begin{eqnarray*}
\begin{aligned}
    \frac{\partial x^i}{\partial \hat{x}^s} &= \left.\frac{\partial x^i}{\partial \hat{x}^s}\right\rvert_{\varepsilon=0}+\left.\frac{d}{d\varepsilon} \frac{\partial x^i}{\partial \hat{x}^k}\right\rvert_{\varepsilon=0}\varepsilon+o(\varepsilon^2)\\
    &= \delta_{is}+\frac{\partial}{\partial\hat{x}^s}\left(\left.\frac{d\Psi^i _{-\varepsilon}}{d\varepsilon}\right\rvert_{\varepsilon=0}\right)\varepsilon+o(\varepsilon^2)\\
    &= \delta_{is}-\frac{\partial \xi^i}{\partial x^s}\varepsilon+o(\varepsilon^2).\quad\quad\quad\quad\quad\quad
\end{aligned}
\end{eqnarray*}
So, \eqref{anali01} becomes
\begin{eqnarray*}
    \hat{g}_{sl}&=&g_{i,j}\left(\delta_{is}-\frac{\partial \xi^i}{\partial x^s}\varepsilon+o(\varepsilon^2)\right)\left(\delta_{jl}-\frac{\partial \xi^j}{\partial x^l}\varepsilon+o(\varepsilon^2)\right)\\
    &=& g_{sl}-\left( g_{si}\frac{\partial \xi^i}{\partial x^l}+g_{il}\frac{\partial \xi^i}{\partial x^s}\right)\varepsilon+o(\varepsilon^2).\quad\quad\quad\quad\quad\,
\end{eqnarray*}
Therefore, as we have claimed, the symmetries $X_{k+2}$ merely depict the covariance of the Ricci flow.
 \end{itemize}
 The conjecture functions as a necessary condition: Indeed the Ricci flow admits \textit{at least} the Lie algebra of symmetries spanned by \eqref{simericciflowConj} To prove the conjecture we need to show that it is also sufficient, that is that the Ricci flow admits no more Lie symmetries than the ones already found. 

\begin{proof}
Let us suppose that
\begin{equation*}
X = \xi^t(t,x_1,\dots, x_n, g) \frac{\partial}{\partial t} + \xi^\mu(t,x_1,\dots, x_n, g) \frac{\partial}{\partial x_\mu} + \eta_{(\mu\nu)}(t,x_1,\dots, x_n, g)\frac{\partial}{\partial g_{\mu\nu}},
\end{equation*}
is the  infinitesimal generator of a Lie point symmetry of the Ricci flow. Observe that the second summation  is restricted so that the dependent variables are not counted twice.

Using this generic form we can arrive at the determining equations. At this point, we employ the idea of Marchildon~\cite{marchildon1998lie} to solve them. First, we write the Ricci Flow as
\begin{equation*}
   2 R_{\alpha\beta}+\partial_t g_{\alpha\beta}=0,
\end{equation*}
where
\begin{eqnarray}\label{riccorex}
\begin{aligned}
R_{\alpha \beta}=& \frac{1}{2} g^{\gamma \delta}\left\{-\partial_{\gamma} \partial_{\delta} g_{\alpha \beta}-\partial_{\alpha} \partial_{\beta} g_{\gamma \delta}+\partial_{\beta} \partial_{\delta} g_{\alpha \gamma}+\partial_{\alpha} \partial_{\gamma} g_{\delta \beta}\right\} \\
&+g^{\gamma \delta} g^{\tau \rho}\left\{\Gamma_{\tau \gamma \alpha} \Gamma_{\rho \delta \beta}-\Gamma_{\tau \gamma \delta} \Gamma_{\rho \alpha \beta}\right\},
\end{aligned}
\end{eqnarray}
and $\Gamma_{\gamma\alpha}^\beta=\Gamma_{\tau\gamma\alpha}g^{\tau\beta}$ with
\begin{equation*}
    \Gamma_{\tau\gamma\alpha}=\frac{1}{2}(\partial_{\alpha}g_{\tau\gamma}+\partial_\gamma g_{\tau\alpha}-\partial_\tau g_{\gamma\alpha}).
\end{equation*}
At this point is advantageous --- for the sake of clarity --- to define the tensor
\begin{equation*}
X^{\mu \nu \kappa \lambda} :=-\frac{\partial g^{\mu \nu}}{\partial g_{\kappa \lambda}}= \begin{cases}g^{\mu \kappa} g^{\nu \lambda}+g^{\mu \lambda} g^{\nu \kappa} & \text { if } \kappa \neq \lambda, \\ g^{\mu \kappa} g^{\nu \lambda} & \text { if } \kappa=\lambda.\end{cases}
\end{equation*}
Since $X^{\mu\nu k\lambda}=X_a^{\nu k \lambda}g^{\mu a}=X_{ab}^{k \lambda}g^{b\nu}g^{\mu a}$, we have that
\begin{equation}\label{definx3}
X_{\mu}^{\nu\kappa\lambda}= \begin{cases}\delta_{\mu}^{\kappa} g^{\nu\lambda}+\delta_{\mu}^{\lambda} g^{\nu\kappa} & \text { if } \kappa \neq \lambda, \\ \delta_{\mu}^{\kappa} g^{\nu\lambda} & \text { if } \kappa=\lambda.\end{cases}
\end{equation}
and
\begin{equation*}
X_{\mu \nu}^{\kappa \lambda}= \begin{cases}\delta_{\mu}^{\kappa} \delta_{\nu}^{\lambda}+\delta_{\mu}^{\lambda} \delta_{\nu}^{\kappa} & \text { if } \kappa \neq \lambda, \\ \delta_{\mu}^{\kappa} \delta_{\nu}^{\lambda} & \text { if } \kappa=\lambda.\end{cases}
\end{equation*}
In particular
\begin{eqnarray}\label{dermetex}
\begin{aligned}
\frac{\partial g_{\mu\nu}}{\partial g_{k\lambda}}&=X_{\mu\nu}^{k\lambda},\\
\frac{\partial(\partial_{\gamma}\partial_\delta g_{\alpha\beta})}{\partial(\partial_k\partial_{\lambda}g_{\mu\nu})}&=X_{\gamma\delta}^{k\lambda}X_{\alpha\beta}^{\mu\nu},\\
2\frac{\partial \Gamma_{\tau\gamma\alpha}}{\partial(\partial_kg_{\mu\nu})}&=\delta^{k}_{\alpha}X_{\tau\gamma}^{\mu\nu}+\delta^{k}_{\gamma}X_{\tau\alpha}^{\mu\nu}-\delta^{k}_{\tau}X^{\mu\nu}_{\gamma\alpha}.
\end{aligned}
\end{eqnarray}
We need to express the partial derivatives of the Ricci tensor with respect to $g_{\mu\nu}$ and its partial derivatives. In what follows, we denote fixed indices with a hat, $\hat{}$. From~\eqref{riccorex},~\eqref{dermetex} we get
\begin{eqnarray*}
\begin{aligned}
 \frac{\partial R_{\alpha \beta}}{\partial\left(\partial_{\kappa} \partial_{\lambda} g_{\mu \nu}\right)} &= \frac{\partial }{\partial\left(\partial_{\kappa} \partial_{\lambda} g_{\mu \nu}\right)}\left( \frac{g^{\gamma \delta}}{2}\left\{-\partial_{\gamma} \partial_{\delta} g_{\alpha \beta}-\partial_{\alpha} \partial_{\beta} g_{\gamma \delta}+\partial_{\beta} \partial_{\delta} g_{\alpha \gamma}+\partial_{\alpha} \partial_{\gamma} g_{\delta \beta}\right\}\right)\\
 &=\frac{1}{2} g^{\gamma \delta}\left\{-X_{\gamma \delta}^{\kappa \lambda} X_{\alpha \beta}^{\mu \nu}-X_{\alpha \beta}^{\kappa \lambda} X_{\gamma \delta}^{\mu \nu}+X_{\delta \beta}^{\kappa \lambda} X_{\gamma \alpha}^{\mu \nu}+X_{\gamma \alpha}^{\kappa \lambda} X_{\delta \beta}^{\mu \nu}\right\},
 \end{aligned}
 \end{eqnarray*}
 \begin{eqnarray*}
 \begin{aligned}
 \frac{\partial R_{\alpha \beta}}{\partial\left(\partial_{\kappa} g_{\mu \nu}\right)} &=\frac{\partial}{\partial\left(\partial_{\kappa} g_{\mu \nu}\right)}\left(g^{\gamma \delta} g^{\tau \rho}\left\{\Gamma_{\tau \gamma \alpha} \Gamma_{\rho \delta \beta}-\Gamma_{\tau \gamma \delta} \Gamma_{\rho \alpha \beta}\right\}\right)\\
 &=\frac{1}{2} g^{\gamma \delta} g^{\tau \rho}\left\{\left[\delta_{\alpha}^{\kappa} X_{\tau \gamma}^{\mu \nu}+\delta_{\gamma}^{\kappa} X_{\tau \alpha}{ }^{\mu \nu}-\delta_{\tau}^{\kappa} X_{\gamma \alpha}{ }^{\mu \nu}\right] \Gamma_{\rho \delta \beta}\right.\\
&+\left[\delta_{\beta}^{\kappa} X_{\rho \delta}{ }^{\mu \nu}+\delta_{\delta}^{\kappa} X_{\rho \beta}^{\mu \nu}-\delta_{\rho}^{\kappa} X_{\delta \beta}{ }^{\mu \nu}\right] \Gamma_{\tau \gamma \alpha} \\
&-\left[\delta_{\delta}^{\kappa} X_{\tau \gamma}{ }^{\mu \nu}+\delta_{\gamma}^{\kappa} X_{\tau \delta}^{\mu \nu}-\delta_{\tau}^{\kappa} X_{\gamma \delta}{ }^{\mu \nu}\right] \Gamma_{\rho \alpha \beta} \\
&\left.-\left[\delta_{\beta}^{\kappa} X_{\rho \alpha}{ }^{\mu \nu}+\delta_{\alpha}^{\kappa} X_{\rho \beta}{ }^{\mu \nu}-\delta_{\rho}^{\kappa} X_{\alpha \beta}{ }^{\mu \nu}\right] \Gamma_{\tau \gamma \delta}\right\},
\end{aligned}
\end{eqnarray*}
\begin{eqnarray*}
\begin{aligned}
\frac{\partial R_{\alpha \beta}}{\partial g_{\mu \nu}} &=\frac{1}{2}\left\{\partial_{\gamma} \partial_{\delta} g_{\alpha \beta}+\partial_{\alpha} \partial_{\beta} g_{\gamma \delta}-\partial_{\delta} \partial_{\beta} g_{\gamma \alpha}-\partial_{\gamma} \partial_{\alpha} g_{\delta \beta}\right\} X^{\gamma \delta \mu \nu} \\
&-\left\{\Gamma_{\tau \gamma \alpha} \Gamma_{\rho \delta \beta}-\Gamma_{\tau \gamma \delta} \Gamma_{\rho \alpha \beta}\right\}\left\{g^{\gamma \delta} X^{\tau \rho \mu \nu}+g^{\tau \rho} X^{\gamma \delta \mu \nu}\right\}.
\end{aligned}
\end{eqnarray*}
See~\cite{marchildon1998lie}. Note that for any symmetric tensor $A$, we have
\begin{equation}\label{tenxi4ind}
    A_{(\mu\nu)}X_{\gamma\delta}^{\mu\nu}=A_{\gamma\delta}.
\end{equation}

The second prolongation of $X$ is
\begin{eqnarray*}
	\begin{aligned}
    X^{(2)}=&\xi^{\mu}\frac{\partial}{\partial x^\mu}+\xi^t\frac{\partial}{\partial t}+\eta_{(\mu\nu)}\frac{\partial}{\partial g_{\mu\nu}}+\eta_{(\mu\nu)k}\frac{\partial}{\partial(\partial_k g_{\mu\nu})}+\eta_{(\mu\nu)t}\frac{\partial}{\partial(\partial_t g_{\mu\nu})}\\
    &+\eta_{(\mu\nu)(k\lambda)}\frac{\partial}{\partial (\partial_k\partial_\lambda g_{\mu\nu})}.
    \end{aligned}
\end{eqnarray*}
Therefore, to obtain the linearization conditions, we need first to expand the equations
\begin{eqnarray}\label{condlineari}
\begin{aligned}
X^{(2)}\Big(R_{\alpha \beta}+&\frac{1}{2}\partial_tg_{\alpha \beta}\Big)=\frac{1}{2} \eta_{\alpha \beta t}+\frac{\eta^{\gamma \delta}}{2} \left\{\partial_{\gamma} \partial_{\delta} g_{\alpha \beta}+\partial_{\alpha} \partial_{\beta} g_{\gamma \delta}-\partial_{\delta} \partial_{\beta} g_{\gamma \alpha}-\partial_{\gamma} \partial_{\alpha} g_{\delta \beta}\right\} \\
&-\left\{\Gamma_{\tau \gamma \alpha} \Gamma_{\rho \delta \beta}-\Gamma_{\tau \gamma \delta} \Gamma_{\rho \alpha \beta}\right\}\left\{g^{\gamma \delta} \eta^{\tau \rho}+g^{\tau \rho} \eta^{\gamma \delta}\right\} \\
&+\frac{1}{2} g^{\gamma \delta} g^{\tau \rho}\left\{\left[\eta_{\tau \gamma \alpha}+\eta_{\tau \alpha \gamma}-\eta_{\gamma \alpha \tau}\right] \Gamma_{\rho \delta \beta}+\left[\eta_{\rho \delta \beta}+\eta_{\rho \beta \delta}-\eta_{\delta \beta \rho}\right] \Gamma_{\tau \gamma \alpha}\right.\\
&\left.-\left[\eta_{\tau \gamma \delta}+\eta_{\tau \delta \gamma}-\eta_{\gamma \delta \tau}\right] \Gamma_{\rho \alpha \beta}-\left[\eta_{\rho \alpha \beta}+\eta_{\rho \beta \alpha}-\eta_{\alpha \beta \rho}\right] \Gamma_{\tau \gamma \delta}\right\} \\
&+\frac{1}{2} g^{\gamma \delta}\left\{-\eta_{\alpha \beta \gamma \delta}-\eta_{\gamma \delta \alpha \beta}+\eta_{\gamma \alpha \delta \beta}+\eta_{\delta \beta \gamma \alpha}\right\}=0,
\end{aligned}
\end{eqnarray}
where
\begin{equation*}
    \eta^{\gamma\delta}=\sum_{a\leq b}\eta_{ab}X^{\gamma\delta a b}=\eta_{(ab)}X^{\gamma\delta a b}.
\end{equation*}
The  functions $\eta_{\tau\gamma\alpha}$, $\eta_{\tau\gamma t}$ and $\eta_{\alpha\beta\gamma\delta}$ are given by the next expressions
\begin{eqnarray*}
\begin{aligned}
    \eta_{\tau\gamma\alpha}&=D_{\alpha}(\eta_{\tau\gamma}-\xi^\sigma\partial_\sigma g_{\tau\gamma}-\xi^{t}\partial_tg_{\tau\gamma})+\xi^\sigma \partial_\alpha \partial_\sigma g^{\tau\gamma}+\xi^t \partial_\alpha \partial_t g^{\tau\gamma},\\
    \eta_{\tau\gamma t}&=D_{t}(\eta_{\tau\gamma}-\xi^\sigma\partial_\sigma g_{\tau\gamma}-\xi^{t}\partial_tg_{\tau\gamma})+\xi^\sigma \partial_t \partial_\sigma g^{\tau\gamma}+\xi^t \partial_t \partial_t g^{\tau\gamma},\\
    \eta_{\alpha\beta\gamma\delta}&=D_{\gamma}D_{\delta}(\eta_{\alpha\beta}-\xi^\sigma\partial_\sigma g_{\alpha\beta}-\xi^t \partial_t g_{\alpha\beta})+\xi^\sigma \partial_\gamma \partial_\delta\partial_\sigma g_{\alpha\beta}+\xi^t\partial_\gamma\partial_\delta\partial_t g_{\alpha\beta},
    \end{aligned}
\end{eqnarray*}
see Theorem 2.36, Olver~\cite{olver2000applications}. Here $D_\alpha$ and $D_t$ denote the total derivative operator 
with respect to $x^\alpha$ and $t$, that is
\begin{eqnarray*}
\begin{aligned}
D_{\alpha}&=\partial_{\alpha}+\partial_{\alpha} g_{(\mu \nu)} \frac{\partial}{\partial g_{\mu \nu}}+\partial_{\alpha} \partial_{\kappa} g_{(\mu \nu)} \frac{\partial}{\partial\left(\partial_{\kappa} g_{\mu \nu}\right)}+\partial_{\alpha} \partial_{(\kappa} \partial_{\lambda)} g_{(\mu \nu)} \frac{\partial}{\partial\left(\partial_{\kappa} \partial_{\lambda} g_{\mu \nu}\right)},\\
D_{t}&=\partial_{t}+\partial_{t} g_{(\mu \nu)} \frac{\partial}{\partial g_{\mu \nu}}+\partial_{t} \partial_{\kappa} g_{(\mu \nu)} \frac{\partial}{\partial\left(\partial_{\kappa} g_{\mu \nu}\right)}+\partial_{t} \partial_{(\kappa} \partial_{\lambda)} g_{(\mu \nu)} \frac{\partial}{\partial\left(\partial_{\kappa} \partial_{\lambda} g_{\mu \nu}\right)}.
\end{aligned}
\end{eqnarray*}
Hence
\begin{eqnarray}\label{pron1ab}
\begin{aligned}
   \eta_{\tau \gamma \alpha}=&\partial_{\alpha} \eta_{\tau \gamma}-\left[\partial_{\sigma} g_{\tau \gamma}\right] \partial_{\alpha} \xi^{\sigma}-\left[\partial_{t} g_{\tau \gamma}\right] \partial_{\alpha} \xi^{t}+\partial_{\alpha} g_{(\mu \nu)} \frac{\partial \eta_{\tau \gamma}}{\partial g_{\mu \nu}}\\
   &-\left[\partial_{\alpha} g_{(\mu \nu)}\right] \partial_{\sigma} g_{\tau \gamma} \frac{\partial \xi^{\sigma}}{\partial g_{\mu \nu}}-\left[\partial_{\alpha} g_{(\mu \nu)}\right] \partial_{t} g_{\tau \gamma} \frac{\partial \xi^{t}}{\partial g_{\mu \nu}},\\
      \eta_{\tau \gamma t}=&\partial_{t} \eta_{\tau \gamma}-\left[\partial_{\sigma} g_{\tau \gamma}\right] \partial_{t} \xi^{\sigma}-\left[\partial_{t} g_{\tau \gamma}\right] \partial_{t} \xi^{t}+\partial_{t} g_{(\mu \nu)} \frac{\partial \eta_{\tau \gamma}}{\partial g_{\mu \nu}}\\
   &-\left[\partial_{t} g_{(\mu \nu)}\right] \partial_{\sigma} g_{\tau \gamma} \frac{\partial \xi^{\sigma}}{\partial g_{\mu \nu}}-\left[\partial_{t} g_{(\mu \nu)}\right] \partial_{t} g_{\tau \gamma} \frac{\partial \xi^{t}}{\partial g_{\mu \nu}},
\end{aligned}
\end{eqnarray}
and
\begin{eqnarray}\label{pron2ab}
\begin{aligned}
   \eta_{\alpha \beta \gamma \delta}=&\partial_{\gamma} \partial_{\delta} \eta_{\alpha \beta}-\partial_\gamma\partial_\delta \xi^\sigma(\partial_\sigma g_{\alpha\beta})-\partial_\gamma\partial_\delta \xi^t (\partial_tg_{\alpha\beta})+(\partial_\delta g_{(\mu\nu)})\partial_\gamma\left(\frac{\partial\eta_{\alpha\beta}}{\partial g_{\mu\nu}}\right)\\
   &+\partial_\gamma g_{(\mu\nu)}\partial_\delta\left(\frac{\partial \eta_{\alpha\beta}}{\partial g_{\mu\nu}}\right)-(\partial_\gamma g_{(\mu\nu)})(\partial_\sigma g_{\alpha\beta})\partial_\delta\left(\frac{\partial\xi^\sigma}{\partial g_{\mu\nu}}\right)\\
   &-(\partial_\delta g_{(\mu\nu)})(\partial_\sigma g_{\alpha\beta})\partial_\gamma\left(\frac{\partial \xi^\sigma}{\partial g_{\mu\nu}}\right)-(\partial_\gamma g_{(\mu\nu)})(\partial_t g_{\alpha\beta})\partial_\delta\left(\frac{\partial\xi^t}{\partial g_{\mu\nu}}\right)\\
   &-(\partial_\delta g_{(\mu\nu)})(\partial_t g_{\alpha\beta})\partial_\gamma\left(\frac{\partial\xi^t}{\partial g_{\mu\nu}}\right)+(\partial_\delta g_{(ab)})(\partial_\gamma g_{(\mu\nu)})\frac{\partial^2\eta_{\alpha\beta}}{\partial g_{\mu\nu}\partial g_{ab}}\\
   &-(\partial_\delta g_{(ab)})(\partial_\gamma g_{(\mu\nu)})(\partial_\sigma g_{\alpha\beta})\frac{\partial^2\xi^\sigma}{\partial g_{\mu\nu}\partial g_{ab}}-(\partial_\delta g_{(ab)})(\partial_\gamma g_{(\mu\nu)})(\partial_t g_{\alpha\beta})\frac{\partial^2\xi^t}{\partial g_{\mu\nu}\partial g_{ab}}\\
   &-(\partial_\delta \partial_\sigma g_{\alpha\beta})\partial_\gamma\xi^\sigma-(\partial_\gamma\partial_\sigma g_{\alpha\beta})\partial_\delta \xi^\sigma-(\partial_\delta\partial_t g_{\alpha\beta})\partial_\gamma\xi^t-(\partial_\gamma\partial_t g_{\alpha\beta})\partial_\delta \xi^t\\
   &+(\partial_\gamma\partial_\delta g_{(\mu\nu)})\frac{\partial\eta_{\alpha\beta}}{\partial g_{\mu\nu}}-(\partial_\gamma g_{(\mu\nu)})(\partial_\delta\partial_\sigma g_{\alpha\beta})\frac{\partial\xi^\sigma}{\partial g_{\mu\nu}}-(\partial_\delta g_{(\mu\nu)})(\partial_\gamma\partial_\sigma g_{\alpha\beta})\frac{\partial\xi^\sigma}{\partial g_{\mu\nu}}\\
   &-(\partial_\gamma g_{(\mu\nu)})(\partial_\delta\partial_t g_{\alpha\beta})\frac{\partial\xi^t}{\partial g_{\mu\nu}}-(\partial_\delta g_{(\mu\nu)})(\partial_\gamma\partial_t g_{\alpha\beta})\frac{\partial\xi^t}{\partial g_{\mu\nu}}\\
   &-(\partial_\sigma g_{\alpha\beta})(\partial_\delta\partial_\gamma g_{(\mu\nu)})\frac{\partial\xi^\sigma}{\partial g_{\mu\nu}}-(\partial_tg_{\alpha\beta})(\partial_\delta\partial_\gamma g_{(\mu\nu)})\frac{\partial\xi^t}{\partial g_{\mu\nu}}.
\end{aligned}
\end{eqnarray}
Finally, we obtain the linearized symmetry conditions --- having in mind that the metric must satisfy also the Ricci flow, $2R_{\alpha\beta}+\partial_t g_{\alpha\beta}=0$ by substituting~\eqref{pron1ab}  and~\eqref{pron2ab} into~\eqref{condlineari}.

The strategy that we will abide by is the following: From the linearized symmetry conditions, we will extract the determining equations starting with the terms involving derivatives of the metric tensor that do not appear in the Ricci flow equation, then continue with terms that involve higher order derivatives of the metric and finally those involving the derivatives of the low order metric tensor since the corresponding determining equations are usually much easier to solve in this order.

Using the expressions of $X^{\mu\nu\kappa\lambda}$, $X_{\mu}^{\nu\kappa\lambda}$ and $X^{\mu\nu}_{\kappa\lambda}$ we can rewrite certain terms, for instance
\begin{eqnarray}\label{ejereneq}
\begin{aligned}
    g^{\gamma\delta}\frac{\partial\xi^t}{\partial g_{(\mu\nu)}}\partial_t g_{\gamma\delta}\partial_\alpha\partial_\beta g_{\mu\nu}&=\partial_tg_{(cd)}(\partial_a\partial_b g_{(\mu\nu)})\frac{\partial\xi^t}{\partial g_{\mu\nu}}\delta^\alpha_a \delta_b^\beta X^{cd}_{\gamma\delta}g^{\gamma\delta}\\
    &=\partial_tg_{(cd)}(\partial_a\partial_b g_{(\mu\nu)})\frac{\partial\xi^t}{\partial g_{\mu\nu}}\delta^\alpha_a \delta_b^\beta G^{cd},
\end{aligned}
\end{eqnarray}
where 
\begin{equation*}
G^{\rho \sigma}=\left\{\begin{array}{cl}
g^{\rho \sigma} & \text { if } \rho=\sigma \\
2 g^{\rho \sigma} & \text { if } \rho \neq \sigma.
\end{array}\right.
\end{equation*}
Renaming the indexes on the right-hand side of~\eqref{ejereneq} as $c\rightarrow\rho$, $a\rightarrow\delta$, $d\rightarrow\kappa$ and $b\rightarrow\gamma$ we get
\begin{eqnarray*}
g^{\gamma\delta}\frac{\partial\xi^t}{\partial g_{(\mu\nu)}}\partial_t g_{\gamma\delta}\partial_\alpha\partial_\beta g_{\mu\nu}=\partial_tg_{(\rho\kappa)}(\partial_\delta\partial_\gamma g_{(\mu\nu)})\frac{\partial\xi^t}{\partial g_{\mu\nu}}\delta^\alpha_\delta \delta_\gamma^\beta G^{\rho\kappa}.
\end{eqnarray*}
We are now ready to extract some of the coefficients keeping in mind \eqref{definx3}, \eqref{dermetex} and~\eqref{tenxi4ind}.
\begin{description}
    \item[From the $\partial g \partial\partial g$ terms]  
    Note that the only components in~\eqref{condlineari} that include the term $\partial g\partial\partial g$ are $\eta_{\alpha\beta\gamma\delta}$, $\eta_{\gamma\delta\alpha\beta}$, $\eta_{\gamma\alpha\delta\beta}$ and $\eta_{\delta\beta\gamma\alpha}$, and since the Ricci flow does not involve the terms $\partial_tg\partial_{\kappa}\partial_{\sigma}g$, we have our first family of determinant equations:
\begin{equation}\label{eqpargparparg2}
	\begin{aligned}
 			 g^{\gamma \delta} \frac{\partial \xi^{t}}{\partial g_{(\mu \nu)}}&\left\{\partial_{t} g_{\alpha \beta} \partial_{\delta} \partial_{\gamma} g_{\mu \nu}+\partial_{t} g_{\gamma \delta} \partial_{\alpha} \partial_{\beta} g_{\mu \nu}-\partial_{t} g_{\alpha \gamma} \partial_{\delta} \partial_{\beta} g_{\mu \nu}\right.\\
			 &\ \left.-\partial_{t} g_{\delta \beta} \partial_{\alpha} \partial_{\gamma} g_{\mu \nu}\right\}=0.  
	\end{aligned}
\end{equation}
Rearranging the previous equation until we have the same index in the derivatives of $g$, we get
\begin{equation}\label{eqtseg1}
    \partial_t g_{(\rho\kappa)}(\partial_\delta\partial_\gamma g_{(\mu\nu)})\frac{\partial\xi^t}{\partial g_{\mu\nu}}\left(g^{\gamma\delta} X^{\rho\kappa}_{\alpha\beta}+\delta_\delta^\alpha \delta
    _\gamma^\beta G^{\rho\kappa}-\delta^\beta _\gamma X^{\delta\rho\kappa}_\alpha-\delta_\delta^\alpha X^{\gamma\rho\kappa}_\beta \right)=0.
\end{equation}

\begin{description}
	\item[Case $\kappa=\rho=\beta,\,\delta=\gamma=\alpha,\,\alpha\neq\beta$] From Eq.~\eqref{eqtseg1} we have
\begin{equation*}\label{eqprime}
     \frac{\partial\xi^t}{\partial g_{(\mu\nu)}}\left(g^{\alpha\alpha}X^{\beta\beta}_{\alpha\beta}- X^{\alpha\beta\beta}_\beta\right)=0,
\end{equation*}
Thus
\begin{equation*}
    \frac{\partial\xi^t}{\partial g_{\mu\nu}}g^{\alpha\beta}=0.
\end{equation*}
for all $\mu\leq\nu$. Therefore $\xi^t$ only depends of $x$ and $t$.	  
\end{description}
Observe that, the terms $\partial_i g\partial_t\partial_j g$ are multiplied by $\frac{\partial\xi^t}{\partial g_{(\mu\nu)}}$, thus they vanish.\\
Since in the Ricci flow there are no terms of the form $\partial_\kappa g\partial_\sigma\partial_\rho g$ we get
\begin{eqnarray}\label{eqpargparparg}
\begin{aligned}
&g^{\gamma \delta} \frac{\partial \xi^{\sigma}}{\partial g_{(\mu \nu)}}\left\{\partial_{\gamma} g_{\mu \nu} \partial_{\delta} \partial_{\sigma} g_{\alpha \beta}+\partial_{\delta} g_{\mu \nu} \partial_{\gamma} \partial_{\sigma} g_{\alpha \beta}+\partial_{\sigma} g_{\alpha \beta} \partial_{\delta} \partial_{\gamma} g_{\mu \nu}\right.
+\partial_{\alpha} g_{\mu \nu} \partial_{\beta} \partial_{\sigma} g_{\gamma \delta}\\
&+\partial_{\beta} g_{\mu \nu} \partial_{\alpha} \partial_{\sigma} g_{\gamma \delta}+\partial_{\sigma} g_{\gamma \delta} \partial_{\alpha} \partial_{\beta} g_{\mu \nu}-\partial_{\delta} g_{\mu \nu} \partial_{\beta} \partial_{\sigma} g_{\alpha \gamma}-\partial_{\beta} g_{\mu \nu} \partial_{\delta} \partial_{\sigma} g_{\alpha \gamma}\\
&-\partial_{\sigma} g_{\alpha \gamma} \partial_{\delta} \partial_{\beta} g_{\mu \nu}\left.-\partial_{\gamma} g_{\mu \nu} \partial_{\alpha} \partial_{\sigma} g_{\delta \beta}-\partial_{\alpha} g_{\mu \nu} \partial_{\gamma} \partial_{\sigma} g_{\delta \beta}-\partial_{\sigma} g_{\delta \beta} \partial_{\alpha} \partial_{\gamma} g_{\mu \nu}\right\}=0.
\end{aligned}
\end{eqnarray}
Rearranging, Eq.~\eqref{eqpargparparg} becomes
\begin{eqnarray}\label{eqxi1}
\begin{aligned}
&\partial_{\gamma} g_{(\mu \nu)} \partial_{\delta} \partial_{\sigma} g_{(\rho \kappa)}\left\{2 g^{\gamma \delta} X_{\alpha \beta}^{\rho \kappa} \frac{\partial \xi^{\sigma}}{\partial g_{\mu \nu}}+g^{\sigma \delta} X_{\alpha \beta}^{\mu \nu} \frac{\partial \xi^{\gamma}}{\partial g_{\rho \kappa}}\right.+\left(\delta_{\alpha}^{\gamma} \delta_{\beta}^{\delta}+\delta_{\alpha}^{\delta} \delta_{\beta}^{\gamma}\right) G^{\rho \kappa} \frac{\partial \xi^{\sigma}}{\partial g_{\mu \nu}}\\
&+\delta_{\alpha}^{\delta} \delta_{\beta}^{\sigma} G^{\mu \nu} \frac{\partial \xi^{\gamma}}{\partial g_{\rho \kappa}} -\delta_{\beta}^{\delta} X_{\alpha}^{\gamma \rho \kappa} \frac{\partial \xi^{\sigma}}{\partial g_{\mu \nu}}-\delta_{\beta}^{\gamma} X_{\alpha}^{\delta \rho \kappa} \frac{\partial \xi^{\sigma}}{\partial g_{\mu \nu}}-\delta_{\beta}^{\sigma} X_{\alpha}^{\delta \mu \nu} \frac{\partial \xi^{\gamma}}{\partial g_{\rho \kappa}} \\
&\left.-\delta_{\alpha}^{\delta} X_{\beta}^{\gamma \kappa \rho} \frac{\partial \xi^{\sigma}}{\partial g_{\mu \nu}}-\delta_{\alpha}^{\gamma} X_{\beta}^{\delta \kappa \rho} \frac{\partial \xi^{\sigma}}{\partial g_{\mu \nu}}-\delta_{\alpha}^{\sigma} X_{\beta}^{\delta \nu \mu} \frac{\partial \xi^{\gamma}}{\partial g_{\rho \kappa}}\right\}=0.
\end{aligned}
\end{eqnarray}
\begin{description}
	\item[Case $\delta=\sigma=\kappa=\hat\kappa, \alpha\neq\hat\kappa, \beta\neq\hat\kappa, \alpha\neq\gamma, \beta\neq\gamma$] From Eq.~\eqref{eqxi1} we have
\begin{equation*}
    0=g^{\hat{\kappa}\hat{\kappa}}X_{\alpha\beta}^{\mu\nu}\frac{\partial\xi^\gamma}{\partial g_{\rho\hat{\kappa}}},
\end{equation*}
that is 
\begin{equation*}
    0=\frac{\partial \xi^\gamma}{\partial g_{\rho\hat{\kappa}}},
\end{equation*}
for all $\gamma$ and $\rho\leq\hat{\kappa}$. Then $\xi^\gamma$ only depends of $x$ and $t$. 
\end{description}
\end{description}

\begin{description}
\item[From the $\partial\partial g$ terms]
Applying the same argument as in~\eqref{eqpargparparg2} and~\eqref{eqxi1}, we have the next set of determining equations
\begin{eqnarray*}
\begin{aligned}
    g^{\gamma\delta}\{&\partial_\delta\partial_tg_{\alpha\beta}\partial_\gamma\xi^t+\partial_\gamma\partial_t g_{\alpha\beta}\partial_\delta \xi^t+\partial_\beta\partial_t g_{\gamma\delta}\partial_\alpha\xi^t+\partial_\alpha\partial_t g_{\gamma\delta}\partial_\beta\xi^t-\partial_\beta\partial_t g_{\gamma\alpha}\partial_\delta \xi^t\\
    &-\partial_\delta\partial_t g_{\gamma\alpha}\partial_\beta \xi^t-\partial_\alpha\partial_t g_{\delta\beta}\partial_\gamma\xi^t-\partial_\gamma\partial_t g_{\delta\beta} \partial_\alpha \xi^t\}=0.
    \end{aligned}
\end{eqnarray*}
Rearranging indices, one can show that the previous equations become
\begin{eqnarray}\label{deterpartx1}
\begin{aligned}
    \partial_\gamma\xi^t\partial_\delta\partial_t g_{(\rho\kappa)}\Big(&2X_{\alpha\beta}^{\rho\kappa}g^{\gamma\delta}+ G^{\rho\kappa}(\delta_\beta^\delta \delta_\alpha^\gamma+\delta_\alpha^\delta \delta_\beta^\gamma)-X_{\alpha}^{\gamma\rho\kappa}\delta_\beta^\delta\\
    &-X_{\alpha}^{\delta\rho\kappa}\delta_\beta^\gamma-X_\beta^{\gamma\rho\kappa}\delta_\alpha^\delta-\delta_{\alpha}^{\gamma}X_\beta^{\delta\rho\kappa}\Big)=0.
    \end{aligned}
\end{eqnarray}
\begin{description}
	\item[Case $\gamma=\delta=\alpha=\beta, \rho=\kappa=\hat\kappa, \hat\kappa\neq\beta$] From Eq.~\eqref{deterpartx1},  it is necessary that
\begin{equation*}
  (\partial_\beta\xi^t)g^{\hat{\kappa}\hat{\kappa}}=0.
\end{equation*}
Thus
\begin{equation*}
    \partial_\beta \xi^t=0,
\end{equation*}
for all $\beta$. Therefore $\xi^t$ depends only on $t$.
\end{description}
\end{description}

Now we group the terms that contain only second spatial derivatives of the metric tensor. However, since we want that $2R_{\alpha\beta}+\partial_t g_{\alpha\beta}=0$, we must also consider the terms that contain $\partial_t g_{\alpha\beta}$. We have
\begin{eqnarray*}
\begin{aligned}
   &+\frac{1}{2}g^{\gamma\delta}(-\partial_\gamma\partial_\delta g_{\alpha\beta}-\partial_\alpha\partial_\beta g_{\gamma\delta}+\partial_\beta\partial_\delta g_{\alpha\gamma}+\partial_\alpha\partial_\gamma g_{\delta\beta})\partial_t\xi^t-\frac{1}{2}g^{\gamma\delta}(-\partial_\gamma\partial_\delta g_{\mu\nu}\\
   &-\partial_\mu\partial_\nu g_{\gamma\delta}+\partial_\nu\partial_\delta g_{\mu\gamma}+\partial_\mu\partial_\gamma g_{\nu\delta } )\frac{\partial\eta_{\alpha\beta}}{\partial g_{(\mu\nu)}}+\frac{1}{2}\eta^{\gamma\delta}(\partial_\gamma\partial_\delta g_{\alpha\beta}+\partial_\alpha\partial_\beta g_{\gamma\delta}\\
   &-\partial_\delta\partial_\beta g_{\gamma\alpha}-\partial_\gamma\partial_\alpha g_{\delta\beta})+\frac{1}{2}g^{\gamma\delta}\left(\partial_\gamma \xi^\sigma\partial_\delta\partial_\sigma g_{\alpha\beta}+\partial_\delta \xi^\sigma\partial_\gamma\partial_\sigma g_{\alpha\beta}-\partial_\gamma\partial_\delta g_{(\mu\nu)}\frac{\partial\eta_{\alpha\beta}}{\partial g_{\mu\nu}}\right)\\
   &+\frac{1}{2}g^{\gamma\delta}\left(\partial_\beta \xi^\sigma\partial_\alpha\partial_\sigma g_{\gamma\delta}+\partial_\alpha\xi^\sigma \partial_\beta\partial_\sigma g_{\gamma\delta}-\partial_\alpha\partial_\beta g_{(\mu\nu)}\frac{\partial\eta_{\gamma\delta}}{\partial g_{\mu\nu}}\right)-\frac{1}{2}g^{\gamma\delta}\Big(\partial_\beta \xi^\sigma\partial_\delta\partial_\sigma g_{\gamma\alpha}\\
   &+\partial_\delta \xi^\sigma \partial_\beta\partial_\sigma g_{\gamma\alpha}-\partial_\delta\partial_\beta g_{(\mu\nu)}\frac{\partial\eta_{\gamma\alpha}}{\partial g_{\mu\nu}}\Big)-\frac{1}{2}g^{\gamma\delta}\Big(\partial_\gamma \xi^\sigma\partial_\alpha\partial_\sigma g_{\delta\beta}+\partial_\alpha\xi^\sigma \partial_\gamma\partial_\sigma g_{\delta\beta}\\
   &-\partial_\gamma\partial_\alpha g_{(\mu\nu)}\frac{\partial\eta_{\delta\beta}}{\partial g_{\mu\nu}}\Big).
   \end{aligned}
\end{eqnarray*}
Rearranging the indices
\begin{eqnarray}\label{indsengder1}
\begin{aligned}
\partial_\gamma\partial_\delta g_{\rho\sigma}\Big(&\delta_\alpha^\rho\delta_\beta^\sigma\eta^{\gamma\delta}+\delta_\gamma^\alpha\delta_\delta^\beta \eta^{\rho\sigma}-\delta_\alpha^\sigma\delta_\beta^\delta \eta^{\rho\gamma}-\delta_\beta^\sigma\delta_\delta^\alpha\eta^{\gamma\rho}+2\delta^\rho_\alpha\delta_\beta^\sigma g^{\kappa\gamma}\partial_\kappa \xi^\delta\\
&-\frac{\partial\eta_{\alpha\beta}}{\partial g_{(\rho\sigma)}}g^{\gamma\delta}+\delta_\gamma^\beta g^{\rho\sigma}\partial_\alpha \xi^\delta+\delta_\gamma^\alpha g^{\rho\sigma}\partial_\beta \xi^\delta-\delta_\gamma^\alpha\delta_\delta^\beta g^{\mu\nu}\frac{\partial\eta_{\mu\nu}}{\partial g_{(\rho\sigma)}}\\
&-\delta_\alpha^\sigma\delta_\gamma^\beta g^{\rho\kappa}\partial_\kappa\xi^\delta-\delta_\alpha^\sigma g^{\rho\gamma} \partial_\beta \xi^\delta+\delta_\delta^\beta g^{\kappa\gamma}\frac{\partial\eta_{\kappa\alpha}}{\partial g_{(\rho\sigma)}}-\delta_\beta^\sigma\delta_\gamma^\alpha g^{\kappa\rho}\partial_\kappa\xi^\delta\\
&-\delta_\beta^\sigma g^{\gamma\rho}\partial_\alpha \xi^\delta+\delta_\delta^\alpha g^{\gamma\kappa}\frac{\partial\eta_{\kappa\beta}}{\partial g_{(\rho\sigma)}}+\{-\delta_\alpha^\rho\delta_\beta^\sigma g^{\gamma\delta}-g^{\rho\sigma} \delta_\gamma^\alpha\delta_\delta^\beta \\
&+\delta_\rho^\alpha\delta_\gamma^\beta g^{\sigma\delta}+\delta_\beta^\sigma\delta_\alpha^\gamma g^{\delta\rho}\}\partial_t\xi^t+g^{\gamma\delta}\frac{\partial \eta_{\alpha\beta}}{\partial g_{(\rho\sigma)}}+g^{\rho\sigma}\frac{\partial\eta_{\alpha\beta}}{\partial g_{(\gamma\delta)}}\\
&-g^{\delta\sigma}\frac{\partial\eta_{\alpha\beta}}{\partial g_{(\gamma\rho)}}-g^{\sigma\delta}\frac{\partial\eta_{\alpha\beta}}{\partial g_{(\rho\gamma)}}\Big).
\end{aligned}
\end{eqnarray}
\begin{description}
	\item[Case $\sigma=\gamma=\delta=\rho=\hat\rho, \hat\rho\neq\alpha\neq\beta$] From Eq.~\eqref{indsengder1} obtain
\begin{equation*}
    \frac{\partial\eta_{\alpha\beta}}{\partial g_{\hat\rho\hat\rho}}g^{\hat\rho\hat\rho}=0,
\end{equation*}
in consequence
\begin{equation}\label{etadifind}
    \eta_{\alpha\beta}=F_{\alpha\beta}(g_{\alpha\rho},g_{\beta\rho},g_{\alpha\alpha},g_{\beta\beta},g_{\kappa\rho},x,t),
\end{equation}
with $\alpha\neq\rho\neq \beta\neq\kappa$.
\end{description}

\begin{description}
	\item[Case $\rho=\sigma=\alpha, \gamma=\delta=\hat\delta, \alpha\neq\hat\delta\neq\beta$] From Eq.~\eqref{indsengder1}, we have the following determining equation
\begin{equation*}\label{equacionrhoneqalpha2}
   -g^{\alpha\hat\delta}\partial_\beta \xi^{\hat\delta}-g^{\hat\delta\alpha}\frac{\partial\eta_{\alpha\beta}}{\partial g_{\alpha\hat\delta}}=0.
\end{equation*}
Using the fact~\eqref{etadifind} we obtain
\begin{equation}\label{eqrhobetadif2}
    \eta_{\alpha\beta}=-\left(\sum_{\kappa\neq \beta,\alpha}g_{\alpha\kappa}\partial_\beta\xi^\kappa\right)+F_{\alpha\beta}(g_{\alpha\beta},g_{\delta\beta},g_{\beta\beta},g_{\alpha\alpha},g_{\kappa\delta},x,t),
\end{equation}
with $\alpha\neq\beta\neq\delta\neq\kappa$.
\end{description}

\begin{description}
	\item[Case $\rho=\sigma=\beta, \gamma=\delta=\hat\delta, \beta\neq\hat\delta\neq\alpha$] From Eq.~\eqref{indsengder1} yields
\begin{equation}\label{equacionrhoneqalpha3}
    -g^{\beta\hat\delta}\partial_\alpha \xi^{\hat\delta}-g^{\hat\delta\beta}\frac{\partial\eta_{\alpha\beta}}{\partial g_{\beta\hat\delta}}=0.
\end{equation}
Having in mind~\eqref{eqrhobetadif2} we obtain from~\eqref{equacionrhoneqalpha3}
\begin{equation}\label{eqalrhodif2}
    \eta_{\alpha\beta}=-\left(\sum_{\kappa\neq \beta,\alpha}g_{\alpha\kappa}\partial_\beta\xi^\kappa+g_{\beta\kappa}\partial_\alpha\xi^\kappa\right)+F_{\alpha\beta}(g_{\alpha\beta},g_{\beta\beta},g_{\alpha\alpha},g_{\kappa\delta},x,t),
\end{equation}
for $\alpha\neq\beta\neq\delta\neq\kappa$.
\end{description}

Note that the only way for $\eta_{\alpha\beta}$ to depend on $g_{\kappa\delta}$ with $\alpha\neq\beta\neq\kappa\neq\delta$ is that $n\geq 4$, our next goal is to prove that $\displaystyle\frac{\partial\eta_{\alpha\beta}}{\partial g_{\kappa\delta}}=0$. We consider for the next step $n\geq 4$.

\begin{description}
	\item[Case $\delta=\gamma,\,\text{and}\, \alpha, \beta\neq \delta$] From Eq.~\eqref{indsengder1} we get
\begin{eqnarray*}
\begin{aligned}
\partial_\delta\partial_\delta g_{\rho\sigma}\Big(&\delta_\alpha^\rho\delta_\beta^\sigma\eta^{\delta\delta}+2\delta^\rho_\alpha\delta_\beta^\sigma g^{\kappa\delta}\partial_\kappa \xi^\delta-g^{\delta\delta}\frac{\partial\eta_{\alpha\beta}}{\partial g_{(\rho\sigma)}}-\delta_\alpha^\sigma g^{\rho\delta} \partial_\beta \xi^\delta-\delta_\beta^\sigma g^{\delta\rho}\partial_\alpha \xi^\delta\\
&-\delta_\alpha^\rho\delta_\beta^\sigma g^{\delta\delta}\partial_t\xi^t+g^{\delta\delta}\frac{\partial \eta_{\alpha\beta}}{\partial g_{(\rho\sigma)}}+g^{\rho\sigma}\frac{\partial\eta_{\alpha\beta}}{\partial g_{\delta\delta}}-g^{\delta\sigma}\frac{\partial\eta_{\alpha\beta}}{\partial g_{(\delta\rho)}}\\
&-g^{\sigma\delta}\frac{\partial\eta_{\alpha\beta}}{\partial g_{(\rho\delta)}}\Big)=0.
\end{aligned}
\end{eqnarray*}
Further, if $\rho\neq\sigma$, and due to the fact that $\partial_\delta\partial_\delta g_{\rho\sigma}=\partial_\delta\partial_\delta g_{\sigma\rho}$, we need to symmetrize the above expression for the indices $\rho$ and $\sigma$, the result is
\begin{eqnarray}\label{eqrhoneqsigm1}
\begin{aligned}
\partial_\delta\partial_\delta g_{\rho\sigma}\Big(&(\delta_\alpha^\rho\delta_\beta^\sigma+\delta_\alpha^\sigma\delta_\beta^\rho)\eta^{\delta\delta}+2(\delta^\rho_\alpha\delta_\beta^\sigma+\delta_\alpha^\sigma\delta_\beta^\rho )g^{\kappa\delta}\partial_\kappa \xi^\delta-g^{\delta\delta}\frac{\partial\eta_{\alpha\beta}}{\partial g_{\rho\sigma}}-(\delta_\alpha^\sigma g^{\rho\delta}\\
&+\delta_\alpha^\rho g^{\sigma\delta} )\partial_\beta \xi^\delta-(\delta_\beta^\sigma g^{\delta\rho}+\delta_\beta^\rho g^{\delta\sigma})\partial_\alpha \xi^\delta-(\delta_\alpha^\rho\delta_\beta^\sigma+\delta_\alpha^\sigma\delta_\beta^\rho) g^{\delta\delta}\partial_t\xi^t\\
&+g^{\delta\delta}\frac{\partial \eta_{\alpha\beta}}{\partial g_{\rho\sigma}}+2g^{\rho\sigma}\frac{\partial\eta_{\alpha\beta}}{\partial g_{\delta\delta}}-\Big(g^{\delta\sigma}\frac{\partial\eta_{\alpha\beta}}{\partial g_{(\delta\rho)}}+g^{\delta\rho}\frac{\partial\eta_{\alpha\beta}}{\partial g_{(\delta\sigma)}}\Big)\\
&-\Big(g^{\sigma\delta}\frac{\partial\eta_{\alpha\beta}}{\partial g_{(\rho\delta)}}+g^{\rho\delta}\frac{\partial\eta_{\alpha\beta}}{\partial g_{(\sigma\delta)}}\Big)\Big).
\end{aligned}
\end{eqnarray}
Finally, if $\sigma=\delta=\hat\delta$, $\rho=\hat\rho$ and $\hat\rho\neq\beta\neq\alpha\neq\hat\delta$, we have the following determining equation
\begin{equation}\label{eqalrhodif2.10}
   \frac{\partial\eta_{\alpha\beta}}{\partial g_{\hat\rho\hat\delta}}g^{\hat\delta\hat\delta}=0.
\end{equation}
From~\eqref{eqalrhodif2} and~\eqref{eqalrhodif2.10} we get
\begin{equation}\label{eqalrhodif2.0}
    \eta_{\alpha\beta}=-\left(\sum_{\kappa\neq \beta,\alpha}g_{\alpha\kappa}\partial_\beta\xi^\kappa+g_{\beta\kappa}\partial_\alpha\xi^\kappa\right)+F_{\alpha\beta}(g_{\alpha\beta},g_{\beta\beta},g_{\alpha\alpha},x,t),
\end{equation}
with $\alpha\neq\beta$.
\end{description}

\begin{description}
	\item[Case $\sigma=\gamma=\delta=\rho=\hat\rho, \alpha=\beta, \text{and}\,\hat\rho\neq\beta$] From Eq.~\eqref{indsengder1}, we have the following determining equation
\begin{equation*}
   g^{\hat\rho\hat\rho}\frac{\partial \eta_{\alpha\alpha}}{\partial g_{\hat\rho\hat\rho}}=0.
\end{equation*}
Therefore
\begin{equation}\label{etaaa01}
    \eta_{\alpha\alpha}=F_{\alpha\alpha}(g_{\alpha\rho},g_{\alpha\alpha},g_{\rho \kappa},x,t),
\end{equation}
with $\rho\neq\alpha\neq \kappa$.
\end{description}

\begin{description}
	\item[Case $\sigma=\rho=\hat\rho, \alpha=\beta, \gamma=\delta=\hat\delta,\text{and}\, \hat\rho\neq\alpha\neq\hat\delta$] From Eq.~\eqref{indsengder1}, we have the following determining equation
\begin{equation}\label{etaaa02}
   -g^{\hat\delta\hat\rho}\frac{\partial \eta_{\alpha\alpha}}{\partial g_{\hat\delta\hat\rho}}=0.
\end{equation}
By~\eqref{etaaa01} and~\eqref{etaaa02}, it follows that
\begin{equation}\label{etaaa04}
    \eta_{\alpha\alpha}=F_{\alpha\alpha}(g_{\alpha\rho},g_{\alpha\alpha},x,t),
\end{equation}
with $\alpha\neq\rho$.
\end{description}

\begin{description}
	\item[Case $\delta=\gamma=\alpha, \sigma=\rho=\hat\rho, \text{and} \, \alpha\neq\hat\rho\neq\beta$] From Eq.~\eqref{indsengder1} and~\eqref{eqalrhodif2}, we have the following determining equation
\begin{equation*}
    g^{\hat\rho\hat\rho}\left(\partial_\beta\xi^\alpha+\frac{\partial\eta_{\alpha\beta}}{\partial g_{\alpha\alpha}}\right)+g^{\alpha\hat\rho}\left(\frac{\partial\eta_{\hat\rho\beta}}{\partial g_{\hat\rho\hat\rho}}+\partial_\beta\xi^{\hat\rho}\right)=0,
\end{equation*}
differentiating this last equation with respect to $g_{\alpha\hat\rho}$
\begin{equation*}
    \frac{\partial g^{\hat\rho\hat\rho}}{\partial g_{\alpha\hat\rho}}\left(\partial_\beta\xi^\alpha+\frac{\partial\eta_{\alpha\beta}}{\partial g_{\alpha\alpha}}\right)+\frac{\partial g^{\alpha\hat\rho}}{\partial g_{\alpha\hat\rho}}\left(\frac{\partial\eta_{\hat\rho\beta}}{\partial g_{\hat\rho\hat\rho}}+\partial_\beta\xi^{\hat\rho}\right)=0.
\end{equation*}
Note that the expressions in parentheses do not depend on $g_{\alpha\rho}$, see~\eqref{eqrhobetadif2}. By manipulating the above equation, we can show that this equation is equivalent to
\begin{equation*}
    (g^{\alpha\alpha}g^{\hat\rho\hat\rho}-g^{\alpha\hat\rho}g^{\alpha\hat\rho})\left(\frac{\partial\eta_{\hat\rho\beta}}{\partial g_{\hat\rho\hat\rho}}+\partial_\beta\xi^{\hat\rho}\right)=0.
\end{equation*}
On the other hand, since $g^{-1}$ is a positive definite matrix, we have that $(g^{\alpha\alpha}g^{\hat\rho\hat\rho}-g^{\alpha\hat\rho}g^{\alpha\hat\rho})>0$. Hence
\begin{equation*}
\frac{\partial\eta_{\hat\rho\beta}}{\partial g_{\hat\rho\hat\rho}}+\partial_\beta\xi^{\hat\rho}=0,
\end{equation*}
for all $\hat\rho\neq\beta$. By~\eqref{eqalrhodif2.0}, we have
\begin{eqnarray}\label{eqalrhodif3}
\begin{aligned}
    \eta_{\alpha\beta}=&-\left(\sum_{\kappa}g_{\alpha\kappa}\partial_\beta\xi^\kappa+g_{\beta\kappa}\partial_\alpha\xi^\kappa\right)+F_{\alpha\beta}(g_{\alpha\beta},x,t)\\
    =&-g_{\alpha\kappa}\partial_\beta\xi^\kappa-g_{\beta\kappa}\partial_\alpha\xi^\kappa+F_{\alpha\beta}(g_{\alpha\beta},x,t).
    \end{aligned}
\end{eqnarray}
with $\alpha\neq\beta$.
\end{description}

\begin{description}
	\item[Case $\rho=\beta=\alpha,\,\gamma=\sigma=\delta=\hat\delta,\text{and}\,\alpha\neq\hat\delta$] From Eq.~\eqref{eqrhoneqsigm1}, we have the following determining equation
\begin{equation}\label{etaaa03}
   -2 g^{\hat\delta\hat\delta}\partial_\alpha\xi^{\hat\delta}-g^{\hat\delta\hat\delta}\frac{\partial\eta_{\alpha\alpha}}{\partial g_{\alpha\hat\delta}}=0.
\end{equation}
Putting \eqref{etaaa03} and~\eqref{etaaa04}, we get
\begin{equation*}
    \eta_{\alpha\alpha}=-2\sum_{\kappa\neq\alpha}g_{\kappa\alpha}\partial_\alpha \xi^\kappa+\hat{F}_{\alpha\alpha}(g_{\alpha\alpha},x,t).
\end{equation*}
We can write
\begin{eqnarray}\label{etaaa05}
\begin{aligned}
    \eta_{\alpha\alpha}=&-2\sum_{\kappa=1}^n g_{\kappa\alpha}\partial_\alpha \xi^\kappa+F_{\alpha\alpha}(g_{\alpha\alpha},x,t)\\
    =&-2 g_{\kappa\alpha}\partial_\alpha \xi^\kappa+F_{\alpha\alpha}(g_{\alpha\alpha},x,t).
    \end{aligned}
\end{eqnarray}
\end{description}

\begin{description}
	\item[Case $\rho=\sigma=\alpha=\beta, \gamma=\delta=\hat\delta,\text{and}\,\hat\delta\neq\alpha$] From Eq.~\eqref{indsengder1} we have the following determining equation
	\begin{equation*}
    \eta^{\hat\delta\hat\delta}+2g^{\kappa\hat\delta}\partial_\kappa\xi^{\hat\delta}-g^{\hat\delta\hat\delta}\partial_t\xi^t=0.
\end{equation*}
Using~\eqref{eqalrhodif3} and~\eqref{etaaa05}, we obtain
\begin{eqnarray}\label{eqetatgg1}
\begin{aligned}
0&=\eta^{\hat\delta\hat\delta}+2g^{\kappa\hat\delta}\partial_\kappa\xi^{\hat\delta}-g^{\hat\delta\hat\delta}\partial_t\xi^t\\
&=\eta_{ab}g^{\hat\delta a}g^{\hat\delta b}+2g^{\kappa\hat\delta}\partial_\kappa\xi^{\hat\delta}-g^{\hat\delta\hat\delta}\partial_t\xi^t\\
&=F_{ab}g^{a\hat\delta}g^{b\hat\delta}-(g_{a\kappa}\partial_b\xi^\kappa+g_{b\kappa}\partial_a\xi^\kappa )g^{a\hat\delta}g^{b\hat\delta}+2g^{\kappa\hat\delta}\partial_\kappa\xi^{\hat\delta}-g^{\hat\delta\hat\delta}\partial_t\xi^t\\
&=F_{ab}g^{a\hat\delta}g^{b\hat\delta}-g^{\hat\delta\hat\delta}\partial_t\xi^t.
\end{aligned}
\end{eqnarray}
Differentiating with respect to $g_{\hat\delta\hat\delta}$ and making elementary simplifications it can be shown that
\begin{equation}\label{eqetatgg2}
   g^{\hat\delta\hat\delta} \frac{\partial F_{\hat\delta\hat\delta}}{\partial g_{\hat\delta\hat\delta}}-F_{ab}g^{a\hat\delta}g^{b\hat\delta}=0.
\end{equation}
Combining~\eqref{eqetatgg1} and~\eqref{eqetatgg2}, we can conclude that
\begin{equation}\label{etaaa06}
    F_{\delta\delta}=\partial_t\xi^t g_{\delta\delta}+\hat{F}_{\delta\delta}(x,t).
\end{equation}
\end{description}

\begin{description}
	\item[Case $\gamma=\delta=\alpha, \rho=\sigma=\beta,\text{and}\,\alpha\neq\beta$] From Eq.~\eqref{indsengder1}, we get the next determining equation
	\begin{eqnarray*}
   -\eta^{\alpha\beta}-(g^{\kappa\beta}\partial_k& \xi^\alpha+g^{\kappa\alpha}\partial_\kappa \xi^\beta)-g^{\alpha\beta}(\partial_\alpha\xi^\alpha-\partial_\beta\xi^\beta)\\
   &+g^{\alpha\beta}\left(\frac{\partial\eta_{\beta\beta}}{\partial g_{\beta\beta}}-\frac{\partial\eta_{\alpha\beta}}{\partial g_{\alpha\beta}}\right)+g^{\alpha\beta}\partial_t\xi^t=0.
\end{eqnarray*}
Using~\eqref{eqalrhodif3} and~\eqref{etaaa05}, this last equation becomes
\begin{equation*}
    0=-F_{ab}g^{\alpha a}g^{\beta b}+g^{\alpha\beta}\frac{\partial F_{\alpha\beta}}{\partial g_{\alpha\beta}}-g^{\alpha\beta}\frac{\partial F_{\beta\beta}}{\partial g_{\beta\beta}}+g^{\alpha\beta}\partial_t\xi^t.
\end{equation*}
By~\eqref{etaaa06}, we have
\begin{equation}\label{detrbdif}
   0=-F_{ab}g^{\alpha a}g^{\beta b}+g^{\alpha\beta}\frac{\partial F_{\alpha\beta}}{\partial g_{\alpha\beta}},
\end{equation}
for all $\alpha\neq\beta$.

Differentiating with respect to $g_{\hat\kappa\hat\kappa}$ with $\hat\kappa\neq\alpha\neq\beta$,
\begin{eqnarray*}
\begin{aligned}
0&=-g^{\alpha\hat\kappa}g^{\beta\hat\kappa}\frac{\partial F_{\hat\kappa\hat\kappa}}{\partial g_{\hat\kappa\hat\kappa}}+F_{ab}\Big[g^{\alpha\hat\kappa}g^{a\hat\kappa}g^{b\beta}+g^{a\alpha}g^{\beta\hat\kappa}g^{b\hat\kappa}\Big]-g^{\alpha\hat\kappa}g^{\beta\hat\kappa}\frac{\partial F_{\alpha\beta}}{\partial g_{\alpha\beta}}\\
&=-\partial_t\xi^t g^{\alpha\hat\kappa}g^{\beta\hat\kappa}+g^{\alpha\hat\kappa}g^{\beta\hat\kappa}\frac{\partial F_{\hat\kappa\beta}}{\partial g_{\hat\kappa\beta}}+g^{\beta\hat\kappa}g^{\alpha\hat\kappa}\frac{\partial F_{\alpha\hat\kappa}}{\partial g_{\alpha\hat\kappa}}-g^{\alpha\hat\kappa}g^{\beta\hat\kappa}\frac{\partial F_{\alpha\beta}}{\partial g_{\alpha\beta}}\\
&=g^{\alpha\hat\kappa}g^{\beta\hat\kappa}\left(-\partial_t\xi^t+\frac{\partial F_{\hat\kappa\beta}}{\partial g_{\hat\kappa\beta}}+\frac{\partial F_{\hat\kappa\alpha}}{\partial g_{\hat\kappa\alpha}}-\frac{\partial F_{\alpha\beta}}{\partial g_{\alpha\beta}}\right),
\end{aligned}
\end{eqnarray*}
for all fixed $\hat\kappa\neq\alpha\neq\beta$. Hence
\begin{equation*}
    0=-\partial_t\xi^t+\frac{\partial F_{\hat\kappa\beta}}{\partial g_{\hat\kappa\beta}}+\frac{\partial F_{\hat\kappa\alpha}}{\partial g_{\hat\kappa\alpha}}-\frac{\partial F_{\alpha\beta}}{\partial g_{\alpha\beta}}.
\end{equation*}
And we can conclude that
\begin{equation*}
    F_{\alpha\beta}=\hat{F}^1_{\alpha\beta}(x,t)g_{\alpha\beta}+\hat{F}^2_{\alpha\beta}(x,t),
\end{equation*}
for all $\alpha\neq\beta$, see \eqref{eqalrhodif3}.

Now, we differentiate equation~\eqref{detrbdif} with respect to $g_{\alpha\beta}$,
\begin{eqnarray*}
\begin{aligned}
0=&-\frac{\partial F_{\alpha\beta}}{\partial g_{\alpha\beta}}(g^{\alpha\alpha}g^{\beta\beta}+g^{\alpha\beta}g^{\alpha\beta})+F_{ab}\Big((g^{\alpha\alpha}g^{a\beta}+g^{\alpha\beta}g^{\alpha a})g^{\beta b}\\
&+g^{\alpha a}(g^{\alpha\beta}g^{\beta b}+g^{\beta\beta}g^{\alpha b})\Big)-(g^{\alpha\alpha}g^{\beta\beta}+g^{\alpha\beta}g^{\beta\alpha})\frac{\partial F_{\alpha\beta}}{\partial g_{\alpha\beta}}\\
=&-2\frac{\partial F_{\alpha\beta}}{\partial g_{\alpha\beta}}(g^{\alpha\alpha}g^{\beta\beta}+g^{\alpha\beta}g^{\alpha\beta})+\frac{\partial F_{\beta\beta}}{\partial g_{\beta\beta}}g^{\beta\beta}g^{\alpha\alpha}+\frac{\partial F_{\alpha\beta}}{\partial g_{\alpha\beta}}g^{\beta\alpha}g^{\alpha\beta}\\
&+\frac{\partial F_{\alpha\beta}}{\partial g_{\alpha\beta}}g^{\alpha\beta}g^{\alpha\beta}+\frac{\partial F_{\beta\beta}}{\partial g_{\beta\beta}}g^{\beta\beta}g^{\alpha\alpha}\\
=&2g^{\alpha\alpha}g^{\beta\beta}\Big(-\frac{\partial F_{\alpha\beta}}{\partial g_{\alpha\beta}}+\partial_t\xi^t\Big),
\end{aligned}
\end{eqnarray*}
for all fixed $\alpha\neq\beta$. Thus
\begin{equation*}
    0=-\frac{\partial F_{\alpha\beta}}{\partial g_{\alpha\beta}}+\partial_t\xi^t.
\end{equation*}
Therefore
\begin{equation}\label{etaab06}
    F_{\alpha\beta}=\partial_t\xi^tg_{\alpha\beta}+\hat{F}_{\alpha\beta}(x,t),
\end{equation}
for all $\beta\neq\alpha$. Putting~\eqref{etaaa06} and~\eqref{etaab06} together, we have that
\begin{equation}\label{etaaa07}
    F_{\alpha\beta}=\partial_t\xi^t g_{\alpha\beta}+\hat{F}_{\alpha\beta}(x,t),
\end{equation}
for all $\alpha,\beta$.
\end{description}

\begin{description}
\item[From the $\partial g$ terms]
For brevity we leave the partial derivatives of the metric in terms of the Christoffel symbols
\begin{eqnarray}\label{eq1cong1}
\begin{aligned}
&-\frac{1}{2}(\Gamma_{\alpha\beta\sigma}+\Gamma_{\beta\alpha\sigma})\partial_t\xi^\sigma+\frac{1}{2}g^{\gamma\delta}g^{\tau\rho}\Big\{(\partial_\alpha\eta_{\tau\gamma}+\partial_\gamma \eta_{\tau\alpha}-\partial_\tau\eta_{\gamma\alpha})\Gamma_{\rho\delta\beta}\\
&+(\partial_\beta \eta_{\rho\delta}+\partial_\delta\eta_{\rho\beta}-\partial_\rho\eta_{\delta\beta})\Gamma_{\tau\gamma\alpha}-(\partial_\delta \eta_{\tau\gamma}+\partial_\gamma \eta_{\tau\delta}-\partial_\tau g_{\gamma\delta})\Gamma_{\rho\alpha\beta}\\
&-(\partial_\beta\eta_{\rho\alpha}+\partial_\alpha\eta_{\rho\beta}-\partial_\rho\eta_{\alpha\beta})\Gamma_{\tau\gamma\delta}\Big\}+\frac{1}{2}g^{\gamma\delta}\Big[\partial_\gamma\partial_\delta\xi^\sigma(\Gamma_{\alpha\beta\sigma}+\Gamma_{\beta\alpha\sigma})\\
&+\partial_\alpha\partial_\beta\xi^\sigma(\Gamma_{\gamma\delta\sigma}+\Gamma_{\delta\gamma\sigma})-\partial_\delta\partial_\beta\xi^\sigma(\Gamma_{\gamma\alpha\sigma}+\Gamma_{\alpha\gamma\sigma})-\partial_\gamma\partial_\alpha\xi^\sigma (\Gamma_{\delta\beta\sigma}+\Gamma_{\beta\delta\sigma})\\
&-2\partial_\gamma\left(\frac{\partial\eta_{\alpha\beta}}{\partial g_{(\mu\nu)}}\right)(\Gamma_{\mu\nu\delta}+\Gamma_{\nu\mu\delta})-\partial_\alpha\left(\frac{\partial\eta_{\gamma\delta}}{\partial g_{(\mu\nu)}}\right)(\Gamma_{\mu\nu\beta}+\Gamma_{\nu\mu\beta})\\
&-\partial_\beta\left(\frac{\partial \eta_{\gamma\delta}}{\partial g_{(\mu\nu)}}\right)(\Gamma_{\mu\nu\alpha}+\Gamma_{\nu\mu\alpha})+\partial_\delta\left(\frac{\partial\eta_{\gamma\alpha}}{\partial g_{(\mu\nu)}}\right)(\Gamma_{\mu\nu\beta}+\Gamma_{\nu\mu\beta})\\
&+\partial_\beta\left(\frac{\partial\eta_{\gamma\alpha}}{\partial g_{(\mu\nu)}}\right)(\Gamma_{\mu\nu\delta}+\Gamma_{\nu\mu\delta})+\partial_\gamma\left(\frac{\partial\eta_{\delta\beta}}{\partial g_{(\mu\nu)}}\right)(\Gamma_{\mu\nu\alpha}+\Gamma_{\nu\mu\alpha})\\
&+\partial_\alpha\left(\frac{\partial\eta_{\delta\beta}}{\partial g_{(\mu\nu)}}\right)(\Gamma_{\mu\nu\gamma}+\Gamma_{\nu\mu\gamma})\Big].
\end{aligned}
\end{eqnarray}

On the other hand, note that
\begin{eqnarray*}
\partial_\gamma\left(\frac{\partial\eta_{\alpha\beta}}{\partial g_{(\mu\nu)}}\right)(\Gamma_{\mu\nu\delta}+\Gamma_{\nu\mu\delta})=\partial_\gamma\left(\frac{\partial F_{\alpha\beta}}{\partial g_{\alpha\beta}}\right)(\Gamma_{\alpha\beta\delta}+\Gamma_{\mu\delta}),
\end{eqnarray*}
see~\eqref{eqalrhodif3} and~\eqref{etaaa05}. Using~\eqref{eqalrhodif3} and~\eqref{etaaa05}, the Eq.~\eqref{eq1cong1} becomes
\begin{eqnarray}\label{eq1cong2}
\begin{aligned}
&-\frac{1}{2}(\Gamma_{\alpha\beta\sigma}+\Gamma_{\beta\alpha\sigma})\partial_t\xi^\sigma+\frac{1}{2}g^{\gamma\delta}g^{\tau\rho}\Big\{(\partial_\alpha \hat{F}_{\tau\gamma}+\partial_\gamma \hat{F}_{\tau\alpha}-\partial_\tau \hat{F}_{\gamma\alpha})\Gamma_{\rho\delta\beta}\\
&+(\partial_\beta \hat{F}_{\rho\delta}+\partial_\delta \hat{F}_{\rho\beta}-\partial_\rho \hat{F}_{\delta\beta})\Gamma_{\tau\gamma\alpha}-(\partial_\delta \hat{F}_{\tau\gamma}+\partial_\gamma \hat{F}_{\tau\delta}-\partial_\tau \hat{F}_{\gamma\delta})\Gamma_{\rho\alpha\beta}\\
&-(\partial_\beta \hat{F}_{\rho\alpha}+\partial_\alpha \hat{F}_{\rho\beta}-\partial_\rho \hat{F}_{\alpha\beta})\Gamma_{\tau\gamma\delta}\Big\}.
\end{aligned}
\end{eqnarray}
Rearranging indices and considering the symmetry $\Gamma_{\lambda\mu\nu}=\Gamma_{\lambda\nu\mu}$,~\eqref{eq1cong2} becomes
\begin{eqnarray}\label{reunaderab}
\begin{aligned}
&\Big\{g^{\tau\lambda}\left(\delta_\nu^\beta g^{\gamma\mu}+\delta_\mu^\beta g^{\gamma\nu}\right)(\partial_\alpha \hat{F}_{\tau\gamma}+\partial_\gamma \hat{F}_{\tau\alpha}-\partial_\tau \hat{F}_{\gamma\alpha})+g^{\lambda\rho}(\delta_{\nu}^\alpha g^{\mu\delta}+\delta_\mu^\alpha g^{\nu\delta})\\
&(\partial_\beta \hat{F}_{\rho\delta}+\partial_\delta \hat{F}_{\rho\beta}-\partial_\rho \hat{F}_{\delta\beta})-g^{\tau\lambda}g^{\gamma\delta}(\delta_\mu^\alpha\delta_\nu^\beta+\delta_\nu^\alpha\delta_\mu^\beta)(\partial_\delta \hat{F}_{\tau\gamma}+\partial_\gamma \hat{F}_{\tau\delta}-\partial_\tau \hat{F}_{\gamma\delta})\\
&-2 g^{\lambda\rho}g^{\mu\nu}\left(\partial_\beta \hat{F}_{\rho\alpha}+\partial_\alpha \hat{F}_{\rho\beta}-\partial_\rho \hat{F}_{\alpha\beta}\right)-\delta_\lambda^\alpha\Big(\delta_\mu^\beta\partial_t\xi^\nu+\delta_\nu^\beta\partial_t\xi^\mu\Big)-\partial_\lambda^\beta\Big(\delta_\mu^\alpha\partial_t\xi^\nu\\
&+\delta_\nu^\alpha\partial_t\xi^\mu\Big)\Big\}\frac{1}{2}\Gamma_{\lambda\mu\nu}.
\end{aligned}
\end{eqnarray}

\begin{description}
	\item[Case $\mu\neq\alpha, \nu\neq\alpha, \mu\neq\beta, \lambda=\hat\lambda,\text{and}\, \nu\neq\beta$] From the previous equation we have the following determining equation
	\begin{equation*}
    0=g^{\hat\lambda\rho}\Big(\partial_\beta \hat{F}_{\rho\alpha}+\partial_\alpha \hat{F}_{\rho\beta}-\partial_\rho \hat{F}_{\alpha\beta}\Big),
\end{equation*}
for $\alpha,\beta,\hat\lambda.$ As $\hat F_{\rho\alpha}$ depend only on $x$ and $t$, see~\eqref{etaaa07}, we can conclude that
\begin{equation*}\label{eqdersinderg0}
    0=\partial_\beta F_{\hat\rho\alpha}+\partial_\alpha F_{\hat\rho\beta}-\partial_{\hat\rho} F_{\alpha\beta},
\end{equation*}
for all $\alpha,\beta,\hat\rho$. Permuting the indices and the fact that $\hat{F}_{\alpha\beta}=\hat{F}_{\beta\alpha}$, we can conclude that
\begin{equation*}
    \partial_\beta\hat{F}_{\hat\rho\alpha}=0,
\end{equation*}
for all $\beta,\hat\rho,\alpha$. In particular, $\hat{F}_{\alpha\beta}$ depends only of $t$:
\begin{equation}\label{etaaa0707}
    F_{\alpha\beta}=\partial_t\xi^t g_{\alpha\beta}+\hat{F}_{\alpha\beta}(t).
\end{equation}
\end{description}

Since $F_{\alpha\beta}$ depends only of $t$ and $g_{\alpha\beta}$, the expression~\eqref{reunaderab} becomes
\begin{equation*}
   \Big\{ -\delta_\lambda^\alpha\Big(\delta_\mu^\beta\partial_t\xi^\nu+\delta_\nu^\beta\partial_t\xi^\mu\Big)-\partial_\lambda^\beta\Big(\delta_\mu^\alpha\partial_t\xi^\nu+\delta_\nu^\alpha\partial_t\xi^\mu\Big)\Big\}\frac{1}{2}\Gamma_{\lambda\mu\nu}.
\end{equation*}

\begin{description}
	\item[Case $\alpha=\lambda, \mu=\beta, \alpha\neq\beta,\,\text{and}\,\nu=\hat\nu$] From the above equation, we get the next determining equation
\begin{equation*}
    \partial_t\xi^{\hat\nu}+\delta_{\hat\nu}^\beta\partial_t\xi^\beta=0,
\end{equation*}
for all $\beta,\hat\nu$. Then
\begin{equation*}
    \partial_t\xi^\beta=0,
\end{equation*}
for all $\beta$. This means that $\xi^\beta$ is only a function of $x$. 
\end{description}
\end{description}

\begin{description}
\item[Terms that do not contain derivatives of the metric tensor] We have the following determining equation
\begin{equation}\label{eqdersinderg}
 0= g_{\alpha\beta}\partial_t\partial_t\xi^t+\partial_t\hat{F}_{\alpha\beta}.
\end{equation}
Hence
\begin{equation*}
    \partial_t\partial_t\xi^t=0.
\end{equation*}
Thus, $\xi^t=a_1 t+a_2$ with $a_1$ and $a_2$ constants. From~\eqref{eqdersinderg}, we can also conclude that
\begin{equation*}
   \partial_t\hat{F}_{\alpha\beta}=0.
\end{equation*}
Therefore ~\eqref{etaaa0707} becomes
\begin{equation*}\label{etaaa0708}
    F_{\alpha\beta}=a_1 g_{\alpha\beta}+C_{\alpha\beta},
\end{equation*}
where $C_{\alpha\beta}$ is constant.
\end{description}

\begin{description}
   \item[From the $\partial g \partial g$ terms] We have
\begin{equation*}
    0=\left\{\Gamma_{\tau \gamma \alpha} \Gamma_{\rho \delta \beta}-\Gamma_{\tau \gamma \delta} \Gamma_{\rho \alpha \beta}\right\}\left\{g^{\gamma \delta} g^{d\rho}g^{c\tau}+g^{\tau \rho} g^{d\delta}g^{c\gamma}\right\}C_{cd},
\end{equation*}
for all fixed $\alpha$ and $\beta$. Rearranging indices
\begin{equation*}
    \Gamma_{\tau\gamma a}\Gamma_{\rho\delta\beta}\Big(\delta_\alpha^a(g^{\gamma\delta}g^{d\rho}g^{c\tau}+g^{\tau\rho}g^{d\delta}g^{c\gamma})-\delta_\alpha^\delta(g^{\gamma a}g^{d\rho}g^{c\tau}+g^{\tau\rho}g^{da}g^{c\gamma})\Big)C_{cd}.
\end{equation*} 
\begin{description}
	\item[Case $\gamma=a=\tau=\alpha, \text{and} \, \rho=\delta=\beta$] From the above equation, we have the following determining equation
	\begin{equation*}
    \Gamma_{\alpha\alpha\alpha}\Gamma_{\beta\beta\beta}\Big(g^{\alpha\beta}g^{d\beta}g^{c\alpha}\Big)C_{cd}=0,
\end{equation*}
thus
\begin{equation*}
    \Big(g^{\alpha\beta}g^{d\beta}g^{c\alpha}\Big)C_{cd}=0.
\end{equation*}
By differentiating it for $\displaystyle\frac{\partial}{\partial g^{\beta\beta}}\frac{\partial}{\partial g^{\alpha\beta}}$, we obtain
\begin{eqnarray*}
    g^{c\alpha}C_{c\beta}+g^{\alpha\beta}C_{\beta\beta}=0.
\end{eqnarray*}
From this equation, we have
\begin{equation*}
    C_{\beta\beta}=0,\,C_{\alpha\beta}=0.
\end{equation*}
\end{description}
\end{description}

Gathering everything together we can conclude that a sufficient condition for $X$ to be Lie point symmetry of the Ricci flow is to be of the form
\begin{equation*}
    X=(a_0+a_1t)\frac{\partial}{\partial t}+\xi^k(x)\frac{\partial}{\partial x^k}-\left( g_{ki}\frac{\partial\xi^k}{\partial x^j}+g_{kj}\frac{\partial\xi^k}{\partial x^i}-a_1g_{ij}\right)\frac{\partial}{\partial g_{(ij)}}.
\end{equation*}

Hence, we have proved the Theorem 1.
\end{proof}

Now, we classify the finite-dimensional algebra of Theorem~\ref{teo1rfn}.
\begin{prop}\label{prop01}
The one dimensional optimal system associated with the Lie algebra $\operatorname{span}\{X_1,X_2\}$ is: $\{X_1, X_2\}$. While the two dimensional optimal system is trivially the whole Lie algebra $\operatorname{span}\{X_1,X_2\}$.
\end{prop}
\begin{proof}
See example $10.1$ in~\cite{hydon2000symmetry}.
\end{proof}
\begin{observation} The Ricci  flow viewed as a mapping of a suitable jet space is analytic. Thus, we have the guarantee that the symmetries obtained through the infinitesimal criterion are the maximum continuous symmetry group, see Corollary 2.74 in~\cite{olver2000applications}.
\end{observation}

\section{The Ricci flow for particular families of metrics.}
\subsection{The Einstein equations}
In this section, we obtain the Lie point symmetries of the vacuum Einstein equations retrieving them from the ones of the Ricci flow.
\begin{prop}\label{einvacuo}
The Lie algebra of the Lie point symmetries of the vacuum Einstein equations in $n-$dimensional manifold is spanned by the vectors:
\begin{eqnarray*}
\begin{aligned}
    X_1=&\sum_{1\leq i\leq j\leq n} g_{ij}\frac{\partial}{\partial g_{ij}},\\
    X_{k+1}=&\xi^k\frac{\partial}{\partial x^k}-\sum_{1\leq i\leq j\leq n} \left( g_{ki}\frac{\partial\xi^k}{\partial x^j}+g_{kj}\frac{\partial\xi^k}{\partial x^i}\right)\frac{\partial}{\partial g_{ij}},
\end{aligned}
\end{eqnarray*}
where $\xi^1, \cdots ,\xi^n$ are arbitrary smooth functions of $\{x^1,\cdots, x^n\}$, $k\in\{1,\cdots,n\}$.
\end{prop}
\begin{proof}
Let
\begin{equation*}
X=c_1X_1+c_2X_2+X_3+\cdots +X_{n+2},
\end{equation*}
where $c_1,c_2\in\mathds{R}$, be a symmetry of the Theorem~\ref{teo1rfn}. We write $X$ in its canonical form:
\begin{equation*}
Q_X=\sum_{1\leq k\leq l\leq n}Q_{kl}\left(t,x,g_{ps},\frac{\partial g_{ps}}{\partial t},\frac{\partial g_{ps}}{\partial x^r}\right)\frac{\partial}{\partial g_{kl}},
\end{equation*}
and  employ the method illustrated in Example~\ref{ex:1}. That being so, we need to consider the following restriction: $\dfrac{\partial g_{ij}}{\partial t} =0$, that is $g_{ij} = g_{ij}(x^1,\cdots,x^n)$. So,  the characteristic takes the form
\begin{equation*}
	\hat{Q}_{kl} = c_2 g_{kl}(x)-\sum_{1\leq s\leq n}\xi^s(x)\frac{\partial g_{kl}}{\partial x^s}-\sum_{1\leq s\leq n}\left(g_{sk}\frac{\partial\xi^s}{\partial x^l}+g_{sl}\frac{\partial\xi^s}{\partial x^k}\right),
\end{equation*}
or equivalently, 
\begin{equation*}
   X=\xi^k\frac{\partial}{\partial x^k}-\sum_{1\leq i \leq j\leq n}\left(c_2 g_{ij}+g_{ki}\frac{\partial \xi^k}{\partial x^j}+g_{kj}\frac{\partial\xi^k}{\partial x^i}\right)\frac{\partial}{\partial g_{ij}},
\end{equation*}
from where we arrive to the desired Proposition.
\end{proof}
\begin{observation}
These symmetries in $n$ dimensions were obtained first by Marchildon in~\cite{marchildon1998lie}. \end{observation}

In the following sections, we apply this method to find the symmetries of the Ricci flow for particular metric families.

\subsection{Ricci Flow on Warped Product Manifolds}\label{subsec:wrpd1}

Given two Riemannian manifolds $\left(B^{n}, g_{B}\right)$, denoted as the base, and $\left(F^{m}, g_{F}\right)$, denoted as the fiber, and a positive smooth warping function $\varphi$ on $B^{n}$, we consider on the product manifold $B^{n} \times F^{m}$ the warped metric $g=\pi_1^{*} g_{B}+(\varphi \circ \pi_1)^{2} \pi_2^{*} g_{F}$, where $\pi_1$ and $\pi_2$ are the natural projections on $B^{n}$ and $F^{m}$, respectively. Under these conditions, the product manifold is called the \textit{warped product} of $B^n$ and $F^m$. 
 
We will study the Ricci flow on the warped product $M=B^n\times F^m$ with 
\begin{equation*}
	 (B^n,g_B)=\left(\mathds{R}^n,\frac{1}{\psi^2(x)}dx^i\otimes dx^i\right),
\end{equation*}
 and by abuse of notation, we write
\begin{equation*}
g=\frac{1}{\psi^2(x)}dx^i\otimes dx^i+\varphi^2(x)g_F(y),
\end{equation*}
where $\psi$ and $\varphi$ being positive smooth functions.
These kinds of products have been extensively studied, see for example, Angenen~\cite{angenent2004example}, Ivey~\cite{ivey1994ricci}, Oliynyk and Woolgar~\cite{oliynyk2007rotationally}, and Ma and Xu~\cite{ma2011ricci}.

We know that the Ricci flow preserves the isometry group of the initial Riemannian manifold. For warped metrics, this means that the Ricci flow preserves the warped structure (see~\cite{oliynyk2007rotationally}). As a consequence, we can write
\begin{equation*}\label{metors}
  g=\frac{1}{\psi^2(x,t)}dx^i\otimes dx^i+\varphi^2(x,t) g_F(y,t).
\end{equation*}
Consider now the smooth fields $W_1, W_2$ on $F^m$, it follows that (see~\cite{bishop1969manifolds,o1983semi})
\begin{eqnarray*}
Ric_M(\tilde W_1,\tilde W_2)=Ric_{g_F}(W_1,W_2)-\left(\frac{\Delta_B \varphi}{\varphi}+\frac{|\nabla_B \varphi|^2}{\varphi^2}(m-1)\right)g(\tilde W_1,\tilde W_2),
\end{eqnarray*}
and by the Ricci flow we have
\begin{eqnarray*}
Ric_M(\tilde W_1,\tilde W_2)=-\frac{\partial\varphi}{\partial t}\varphi g_F(W_1,W_2).
\end{eqnarray*}
Thus
\begin{eqnarray*}
Ric_{g_F}(W_1,W_2)=\left(-\varphi\frac{\partial\varphi}{\partial t}+\varphi^2\left(\frac{\Delta_B \varphi}{\varphi}+\frac{|\nabla_B \varphi|^2}{\varphi^2}(m-1)\right)\right)g_F(W_1,W_2).
\end{eqnarray*}
Finally, we have
\begin{equation*}
  -\varphi\frac{\partial\varphi}{\partial t}+\varphi^2\left(\frac{\Delta_B \varphi}{\varphi}+\frac{|\nabla_B \varphi|^2}{\varphi^2}(m-1)\right)=\mu.
\end{equation*} 
with $\mu$ depending only on t. In particular, $(F^m,g_F)$ must be an Einstein manifold for each $t$.  This motivates us to study the case in which the metric over time is of the form:
\begin{equation}\label{eqro1}
  g=\frac{1}{\psi^2(x,t)}dx^i\otimes dx^i+\varphi^2(x,t) g_{can}.
\end{equation}
where $(F^m,g_{can})$ is an Einstein manifold. 

\begin{teor}\label{teotor}
The Lie algebra of the Lie point symmetries of the Ricci Flow of~\eqref{eqro1}, with $(F^m,g_{can})$ being an Einstein manifold non-homothetic to the euclidean metric, is spanned by:

\begin{eqnarray*}
\begin{aligned}
X_1&=\frac{\partial}{\partial t},\\
X_2&=t\frac{\partial}{\partial t}-\frac{\psi}{2}\frac{\partial}{\partial \psi}+\frac{\varphi}{2}\frac{\partial}{\partial \varphi},\\
X_{k+2}&=\xi^k(x)\frac{\partial}{\partial x}+\psi\frac{\partial\xi^k}{\partial x^k}\frac{\partial}{\partial \psi},
\end{aligned}
\end{eqnarray*}
where $\xi^k(x)$ are smooth functions of $\mathds{R}^n$ with $k\in\{1,\dots,n\}$, such that
\begin{equation*}
    \frac{\partial\xi^i}{\partial x^j}+\frac{\partial \xi^j}{\partial x^i}=0\quad\quad\text{and}\quad\quad\frac{\partial\xi^i}{\partial x^i}-\frac{\partial\xi^j}{\partial x^j}=0,
\end{equation*}
with $i\neq j$.
\end{teor}
\begin{proof}
We begin by taking a symmetry of the Ricci flow 
\begin{equation*}
X=c_1X_1+c_2X_2+X_3+\cdots +X_{n+m+2},
\end{equation*}
where $c_1,c_{2}\in\mathds{R}$. We write $X$ in its canonical form:
\begin{equation*}
Q_X=\sum_{1\leq i\leq j\leq n}Q_{ij}\left(t,x,g_{ps},\frac{\partial g_{ps}}{\partial t},\frac{\partial g_{ps}}{\partial x^r}\right)\frac{\partial}{\partial g_{ij}}.
\end{equation*}

Having in mind~\eqref{eqro1}, we see that we need to apply the method on the restrictions $\displaystyle g_{ii}=\frac{1}{\psi^2(x,t)}$, $g_{i,j}=0$ for $i,j\in\{1,\cdots,n\}$ with $i\neq j$ and $g_{k+n,l+n}=\varphi(x,t)^2 (g_{can})_{ij}$ with $1\leq k\leq l\leq m$. 
Thus, we look for functions $Q_\psi $ and $Q_\varphi$ depending on $\displaystyle x,t, \varphi,\psi,\frac{\partial \varphi}{\partial x}, \frac{\partial \psi}{\partial x}, \frac{\partial \phi}{\partial t}, \frac{\partial \psi}{\partial t}$ such that
\begin{equation*}
Q_X=Q_\varphi\frac{\partial}{\partial \varphi}+Q_\psi\frac{\partial}{\partial \psi},
\end{equation*}
under the restrictions described above we get the conditions
\begin{eqnarray*}
\begin{aligned}
   -\frac{2}{\psi^3}Q_\psi &=Q_X(g_{ii})\Big|_{(x,t,\varphi,\psi)}\\
   &=\left(c_2-2\frac{\partial\xi^i}{\partial x^i}\right)\frac{1}{\psi^2}-(c_1+c_2 t)\left(\frac{-2\psi_t}{\psi^3}\right)-\xi^k\left(\frac{-2\psi_{x^k}}{\psi^3}\right),
\end{aligned}
\end{eqnarray*}
for all $i\in\{1,\dots,n\}$. Then
\begin{equation}\label{eqrefteo1}
    Q_\psi=-\left(\frac{c_2}{2}-\frac{\partial\xi^i}{\partial x^i}\right)\psi-(c_1+c_2t)\psi_t-\xi^k\psi_{x^i}.
\end{equation}
In addition, we have that

\begin{equation}\label{eqcomreftor}
    Q_\psi=-\frac{\psi^3}{2}Q_X(g_{ii})\Big|_{(x,t,\varphi,\psi)}=-\frac{\psi^3}{2}Q_X(g_{jj})\Big|_{(x,t,\varphi,\psi)},
\end{equation}
for $i,j\in\{1,\dots, n\}$. Looking~\eqref{eqrefteo1}, we can see that~\eqref{eqcomreftor} is valid if and only if 
\begin{equation*}\label{cod2torrn}
    \frac{\partial\xi^i}{\partial x^i}=\frac{\partial\xi^j}{\partial x^j},
\end{equation*}
with $i,j\in\{1,\dots, n\}$. For $i,j\in\{1,\dots,n\}$ with $i\neq j$, we have
\begin{eqnarray*}
\begin{aligned}
0=Q_X(g_{ij})\Big|_{(x,t,\varphi,\psi)}=-\frac{1}{\psi^2}\left(\frac{\partial\xi^i}{\partial x^j}+\frac{\partial\xi^j}{\partial x^i}\right),
\end{aligned}
\end{eqnarray*}
so
\begin{equation*}
    \frac{\partial\xi^i}{\partial x^j}+\frac{\partial\xi^j}{\partial x^i}=0,
\end{equation*}
for all $i\neq j$. 

For $i\in\{1,\dots,n\}$ and $l\in\{1,\dots,m\}$. We get
\begin{equation*}
    0=Q_X(g_{i,n+l})\Big|_{(x,t,\varphi,\psi)}=-\left(\frac{1}{\psi^2}\frac{\partial\xi^i}{\partial x^{l+n}}+\varphi^2\frac{\partial\xi^{l+n}}{\partial x^i}\right).
\end{equation*}
In this case, $\partial\xi^i/\partial x^{l+n}$ represents the partial derivative of the function $\xi^i$ with respect to the coordinate $l+n$ of the fiber. On the other hand, as the functions $\xi^i$  depend on the independent variables, we can conclude that
\begin{equation}\label{ojoalpar}
    \xi^i(x^1,\dots,x^n)\quad\text{and}\quad \xi^{l+n}(x^{n+1},\dots,x^{n+m}),
\end{equation}
for all $i\in\{1,\dots,n\}$ and $l\in\{1,\dots,m\}$. For $l,p\in\{1,\dots,m\}$, we obtain
\begin{eqnarray}\label{ojoparc2}
\begin{aligned}
2\varphi Q_\varphi &(g_{can})_{l,p}=Q_X(g_{n+l,n+p})\Big|_{(x,t,\varphi,\psi)}\\
=&\varphi^2\left(c_2 (g_{can})_{lp}-(g_{can})_{kl}\frac{\partial\xi^{n+k}}{\partial x^{n+p}}-(g_{can})_{kp}\frac{\partial\xi^{n+k}}{\partial x^{n+l}}\right)\\
&-2(c_1+c_2 t)\varphi\varphi_t(g_{can})_{lp}-\xi^{n+k}\varphi^2\frac{\partial (g_{can})_{lp}}{\partial x^{n+k}}-2\varphi\varphi_{x_s}\xi^s(g_{can})_{l,p},
\end{aligned}
\end{eqnarray}
in this last equation $k\in\{1,\dots,m\}$ and $s\in\{1,\dots,n\}$. Note that $Q_\varphi$ doesn't depend of $g_{can}$ (the coefficients of the canonical metric are not constant), therefore
\begin{equation*}
    \xi^{n+k}\frac{\partial (g_{can})_{lp}}{\partial x^{n+k}}=0, \quad\text{for}\,k\in\{1,\dots,m\}.
\end{equation*}
Thus, $\xi^{n+k}=0$, for $,k\in\{1,\dots,m\}$ so
\begin{equation*}
    Q_\varphi=\frac{c_2}{2}\varphi-(c_1+c_2t)\varphi_t-\xi^s\varphi_{x_s}.
\end{equation*}
Putting all of the above together
\begin{equation*}
X=\left(c_1+2c_2t\right)\frac{\partial}{\partial t}+\xi^k\frac{\partial}{\partial x^k}+c_2\frac{\partial}{\partial \varphi}+\left(\frac{\partial\xi^k}{\partial x^k}\psi-c_2\psi\right)\frac{\partial}{\partial\psi},
  \end{equation*}
  with $ k\in\{1,\dots,n\}$, $\displaystyle \frac{\partial\xi^i}{\partial x^i}-\frac{\partial\xi^j}{\partial x^j}=0$ and $\displaystyle\frac{\partial\xi^i}{\partial x^j}+\frac{\partial\xi^j}{\partial x^i}=0$ when $i\neq j$.
\end{proof}

Now we turn our attention to the special case where the fiber is the Euclidean space endowed with the canonical metric.

\begin{teor}\label{teotoreuc}
The Lie algebra of the Lie point symmetries of the Ricci Flow of~\eqref{eqro1}, with $F^m$ being $\mathds{R}^m$ endowed with the Euclidean metric, is spanned by:

\begin{eqnarray*}
\begin{aligned}
X_1&=\frac{\partial}{\partial t},\\
X_2&=t\frac{\partial}{\partial t}-\frac{\psi}{2}\frac{\partial}{\partial \psi}+\frac{\varphi}{2}\frac{\partial}{\partial \varphi},\\
X_3&=\varphi\frac{\partial}{\partial \varphi},\\
X_{k+3}&=\xi^k(x)\frac{\partial}{\partial x}+\psi\frac{\partial\xi^k}{\partial x^k}\frac{\partial}{\partial \psi},
\end{aligned}
\end{eqnarray*}
where $\xi^k(x)$ are smooth functions of $\mathds{R}^n$ with $k\in\{1,\dots,n\}$, such that
\begin{equation*}
    \frac{\partial\xi^i}{\partial x^j}+\frac{\partial \xi^j}{\partial x^i}=0\quad\quad\text{and}\quad\quad\frac{\partial\xi^i}{\partial x^i}-\frac{\partial\xi^j}{\partial x^j}=0,
\end{equation*}
with $i\neq j$.
\end{teor}
\begin{proof}
Following the previous proof, we have that $Q_\psi$ is still determined by the expression given by~\eqref{eqrefteo1}, $\xi^i$ with $i\in\{1,\dots,n\}$ continues to satisfy the Cauchy Riemann equations, the Eqs.~\eqref{ojoalpar} are still valid and the Eqs.~\eqref{ojoparc2} are transformed to
\begin{eqnarray*}\label{ojoparc3}
\begin{aligned}
2\varphi Q_\varphi \delta_{l,p}=&Q_X(g_{n+l,n+p})\Big|_{(x,t,\varphi,\psi)}\\
=&\varphi^2\left(c_2 (\delta_{lp}-\frac{\partial\xi^{n+l}}{\partial x^{n+p}}-\frac{\partial\xi^{n+p}}{\partial x^{n+l}}\right)\\
&-2(c_1+c_2 t)\varphi\varphi_t\delta_{lp}-2\varphi\varphi_{x_s}\xi^s\delta_{l,p},
\end{aligned}
\end{eqnarray*}
 for $l,p\in\{1,\dots,m\}$, from which it is possible to conclude that
\begin{equation*}
    Q_\varphi=\frac{c_2}{2}\varphi-(c_1+c_2t)\varphi_t-\xi^s\varphi_{x_s}+c_3\varphi.
\end{equation*}
where $c_3\in\mathds{R}$. Putting all of the above together
\begin{equation*}
X=\left(c_1+2c_2t\right)\frac{\partial}{\partial t}+\xi^k\frac{\partial}{\partial x^k}+(c_2+c_3\varphi)\frac{\partial}{\partial \varphi}+\left(\frac{\partial\xi^k}{\partial x^k}\psi-c_2\psi\right)\frac{\partial}{\partial\psi},
  \end{equation*}
  with $ k\in\{1,\dots,n\}$, $\displaystyle \frac{\partial\xi^i}{\partial x^i}-\frac{\partial\xi^j}{\partial x^j}=0$ and $\displaystyle\frac{\partial\xi^i}{\partial x^j}+\frac{\partial\xi^j}{\partial x^i}=0$ when $i\neq j$.
\end{proof}

The next Corollary is an immediate consequence of the previous Theorem.

\begin{corol}
The Lie algebra of the Lie point symmetries of the Ricci Flow on $\mathds{R}^{n}$ with conformal flow, i.e., 
\begin{equation*}
  g=\frac{1}{\psi^2(x,t)}dx^i\otimes dx^i,
\end{equation*}
 is spanned by:
\begin{eqnarray*}
\begin{aligned}
X_1&=\frac{\partial}{\partial t},\\
X_2&=t\frac{\partial}{\partial t}-\frac{\psi}{2}\frac{\partial}{\partial \psi},\\
X_3&=\xi^k(x)\frac{\partial}{\partial x}+\psi\frac{\partial\xi^k}{\partial x^k}\frac{\partial}{\partial \psi},
\end{aligned}
\end{eqnarray*}
where $\xi^k(x)$ are smooth functions of $\mathds{R}^n$ with $k\in\{1,\dots,n\}$, such that
\begin{equation*}
    \frac{\partial\xi^i}{\partial x^j}+\frac{\partial \xi^j}{\partial x^i}=0\quad\quad\text{and,}\quad\quad\frac{\partial\xi^i}{\partial x^i}-\frac{\partial\xi^j}{\partial x^j}=0,
\end{equation*}
with $i\neq j$.
\end{corol}
\begin{proof}
The result follows immediately through the first part of the proof of Theorem~\ref{teotor}
\end{proof}

\subsection{Ricci Flow on Doubly-Warped Product Manifolds}
We now restrict our attention to flows evolving from a fixed initial metric
that is cohomogenou with respect to the natural $SO(p+1)\times SO(q+1)$ action. That is, we consider the doubly-warped
product metrics on $S^{p+1}\times S^q$ with $p\geq 2$ and $q\geq 2$, i.e.,
\begin{equation*}
    g=\chi(x,t)^2dx\otimes dx+\varphi(x,t)^2g_{S^p _{can}}+\psi(x,t)^2g_{S^q_{can}},
\end{equation*}
where $g_{S^n _{can}}$ denotes the canonical metric of the sphere of dimension $n$.
This kind of metric was extensively studied in~\cite{stolarski2019curvature}.

\begin{teor}\label{dotor}
The Lie algebra of the Lie point symmetries of the Ricci flow on $S^{p+1}\times S^{q}$ for metrics of the form
\begin{equation}\label{ditor2}
   g=\chi^2(x,t)dx\otimes dx+\varphi^2(x,t)g_{S^p_{can}}+\psi^2(x,t)g_{S^q_{can}},
\end{equation}
is spanned by:
\begin{eqnarray*}
\begin{aligned}
X_1 &=\frac{\partial}{\partial t},\\
X_2 &=2t\frac{\partial}{\partial t}+\chi\frac{\partial}{\partial \chi}+\varphi\frac{\partial}{\partial \varphi}+\psi\frac{\partial}{\partial \psi},\\
X_3 &=\xi(x)\frac{\partial}{\partial x}-\chi\frac{\partial\xi}{\partial x}\frac{\partial}{\partial\chi},
\end{aligned}
\end{eqnarray*}
with $\xi(x)$ being an arbitrary smooth function.
\end{teor}
\begin{proof}
Similarly to the proof of Theorem~\ref{teotor} we start by taking a symmetry from Theorem~\ref{teo1rfn}, 
\begin{equation*}
    X=c_1X_1+c_2X_2+\sum_{1\leq k\leq p+q+1}X_{2+k},
\end{equation*}
where $c_1,\cdots,c_{3+p+iq}\in\mathds{R}$ and expressing it in canonical form:
\begin{equation*}
    Q_X=\sum_{1\leq i\leq j \leq 1+p+q}Q_{ij}\frac{\partial}{\partial g_{ij}}.
\end{equation*}
Having in mind~\eqref{ditor2}, we see that we need to apply the method on the restrictions $\displaystyle g_{11}=\chi^2(x,t)$, $g_{1+i_1,1+j_1}=\varphi^2 (g_{S^p_{can}})_{i_1,j_1}$ for $i_1,j_1\in\{1,\cdots,p\}$, $g_{1+p+i_2,1+p+j_2}=\psi^2 (g_{S^q_{can}})_{i_2,j_2}$ for $i_2,j_2\in\{1,\cdots,q\}$  while the rest of the coefficients of the metric are identically zero. Thus, we look for functions $Q_\chi $, $Q_\varphi$ and $Q_\psi$ depending on $\displaystyle x,t$ and of the dependent variables $\chi$, $\varphi$ and $\psi$ as well as their first derivatives, such that
\begin{equation*}\label{camdotorc}
    Q=Q_\varphi\frac{\partial}{\partial\varphi}+Q_\psi\frac{\partial}{\partial\psi}+Q_\chi\frac{\partial}{\partial\chi},
\end{equation*}
under the restrictions described above, proceeding analogously to the proof of Theorem~\ref{teotor}, we get
\begin{eqnarray*}
\begin{aligned}
 Q_\chi&=-\frac{\partial\chi}{\partial t}(c_1+c_2 t)-\frac{\partial\chi}{\partial x}\xi^1(x)+c_2\frac{\chi}{2}-\frac{\chi}{2}\frac{\partial\xi^1}{\partial x},\\
  Q_\varphi&=-(c_1+c_2 t)\frac{\partial\varphi}{\partial t}-\xi^1\frac{\partial\varphi}{\partial x}+\frac{c_2}{2}\varphi,\\
    Q_\psi&=-(c_1+c_2 t)\frac{\partial\psi}{\partial t}-\xi^1\frac{\partial\psi}{\partial x}+\frac{c_2}{2}\psi.
\end{aligned}
\end{eqnarray*}
From this, we can conclude the statement of the theorem.
\end{proof}

\section{Invariant Solutions}

\subsection{Invariant Solutions for Warped Product Manifolds}
First, we need to express explicitly the Ricci flow equations for the metric tensor~\eqref{eqro1}. We know that

\begin{eqnarray*}
\begin{aligned}
Ric_g(\tilde Y,\tilde Z)&=Ric_{g_B}(Y,Z)-\frac{m}{\varphi}\nabla_B^2 \varphi(Y,Z),\\
Ric_g(\tilde Y,\tilde V)&=0,\\
Ric_g(\tilde V,\tilde W)&=Ric_{g_F}(V,W)-\left(\frac{\Delta_B \varphi}{\varphi}+\frac{|\nabla_B \varphi|^2}{\varphi^2}(m-1)\right)g(\tilde V,\tilde W),
\end{aligned}
\end{eqnarray*}
for all $Y,Z\in\mathfrak{L}(B)$ and $V,W\in\mathfrak{L}(F)$, see Proposition $9.106$ in \cite{besse2007einstein} or~\cite{bishop1969manifolds}. On the other hand, consider $(p,q)\in B^n\times F^m$, where $\{\partial_1^B,\dots,\partial_n^B\}$ and $\{\partial_1^F,\dots,\partial_m^F\}$ are coordinate bases of $T_p B$ and $T_q F$ respectively. Then
\begin{eqnarray*}
\begin{aligned}
Ric_g(\tilde \partial^B_i,\tilde \partial^B_j)&=Ric_{g_B}(\partial^B_i,\partial^B_j)-\frac{m}{\varphi}\nabla_B^2 \varphi(\partial^B_i,\partial^B_j).
\end{aligned}
\end{eqnarray*}
From conformal theory we know that
\begin{equation*}
    Ric_B(\partial_i^B,\partial_j^B)=(n-2)\frac{\psi_{ij}}{\psi}+\left(\frac{\Delta\psi}{\psi}-(n-1)\frac{|\nabla \psi|^2}{\psi^2}\right)\delta_{ij},
\end{equation*}
and
\begin{equation*}
  \nabla_B^2 \varphi(\partial^B_i,\partial^B_j)=\varphi_{ij}+\frac{\psi_i\varphi_j+\psi_j\varphi_i}{\psi}-\frac{\langle\nabla\psi,\nabla\varphi\rangle}{\psi}\delta_{ij},
\end{equation*}
see Theorem 1.159 in \cite{besse2007einstein}. Thus
\begin{eqnarray*}
\begin{aligned}
    Ric_g(\tilde \partial^B_i,\tilde \partial^B_j)=&(n-2)\frac{\psi_{ij}}{\psi}+\left(\frac{\Delta\psi}{\psi}-(n-1)\frac{|\nabla \psi|^2}{\psi^2}\right)\delta_{ij}\\
    &-\frac{m}{\varphi}\Big(\varphi_{ij}+\frac{\psi_i\varphi_j+\psi_j\varphi_i}{\psi}-\frac{\langle \nabla\psi,\nabla\varphi\rangle}{\psi}\delta_{ij}\Big).
    \end{aligned}
\end{eqnarray*}
On the other hand
\begin{eqnarray*}
\begin{aligned}
Ric_g(\tilde\partial^F_i,\tilde\partial^F_j)&=Ric_{g_F}(\partial^F_i,\partial^F_j)-\left(\frac{\Delta_B \varphi}{\varphi}+\frac{|\nabla_B \varphi|^2}{\varphi^2}(m-1)\right)g(\tilde \partial^F_i,\tilde \partial^F_j)\\
&=\mu (g_{can})_{ij}-\left(\frac{\Delta_B \varphi}{\varphi}+\frac{|\nabla_B \varphi|^2}{\varphi^2}(m-1)\right)\varphi^2(g_{can})_{ij}.
\end{aligned}
\end{eqnarray*}
Again, from conformal theory, we have
\begin{equation*}
    Ric_g(\tilde\partial^F_i,\tilde\partial^F_j)=\left\{\mu-\psi^2\varphi^2\left(\frac{\Delta\varphi}{\varphi}-(n-2)\frac{\langle\nabla\varphi,\nabla\psi\rangle}{\psi\varphi}+(m-1)\frac{|\nabla\varphi|^2}{\varphi^2}\right)\right\}(g_{can})_{ij}.
\end{equation*}
Therefore
\begin{eqnarray*}
\begin{aligned}
    Ric_g=&\Big((n-2)\frac{\psi_{ij}}{\psi}+\left(\frac{\Delta\psi}{\psi}-(n-1)\frac{|\nabla \psi|^2}{\psi^2}\right)\delta_{ij}-\frac{m}{\varphi}\Big(\varphi_{ij}+\frac{\psi_i\varphi_j+\psi_j\varphi_i}{\psi}\\
    &-\frac{\langle \nabla\psi,\nabla\varphi\rangle}{\psi}\delta_{ij}\Big)\Big)dx_i\otimes dx_j+\Big(\mu-\varphi\psi^2\Big(\Delta\varphi-(n-2)\frac{\langle \nabla\varphi,\nabla\psi\rangle}{\psi}\Big)\\
    &-(m-1)\psi^2|\nabla\varphi|^2\Big)g_{can}.
    \end{aligned}
\end{eqnarray*}
That being done, the family of metrics~\eqref{eqro1} satisfies the Ricci flow equation if and only if the following system is satisfied
\begin{eqnarray}\label{sistricciconfor}
\begin{aligned}
\frac{\psi_t}{\psi^3}\delta_{ij}=&(n-2)\frac{\psi_{ij}}{\psi}+\left(\frac{\Delta\psi}{\psi}-(n-1)\frac{|\nabla \psi|^2}{\psi^2}\right)\delta_{ij}-\frac{m}{\varphi}\Big(\varphi_{ij}+\frac{\psi_i\varphi_j+\psi_j\varphi_i}{\psi}\\
&-\frac{\langle \nabla\psi,\nabla\varphi\rangle}{\psi}\delta_{ij}\Big),\\
-\varphi\varphi_t=&\mu-\varphi\psi^2\Big(\Delta\varphi-(n-2)\frac{\langle \nabla\varphi,\nabla\psi\rangle}{\psi}\Big)-(m-1)\psi^2|\nabla\varphi|^2,
\end{aligned}
\end{eqnarray}
for all $i,j\in{1,\dots,n}$.

Secondly, we use the symmetries found in subsection~\ref{subsec:wrpd1} to reduce the system of equations~\eqref{sistricciconfor}:

Note that $\{X_1, X_2\}$ is the one dimensional optimal system associated to the finite dimensional sub algebra of Theorem~\ref{teotor}, see Proposition~\ref{prop01}. So, we only need to consider, for instance,  the symmetry $X_2$ applying to it the one-parameter group of transformations generated by  $X_1$, in this case, $t\mapsto t+\varepsilon$. That is, we shall employ the symmetry
\begin{equation*}
    X=(1+2kt)\frac{\partial}{\partial t}-k\psi\frac{\partial}{\partial \psi}+k\varphi\frac{\partial}{\partial \varphi},
\end{equation*}
where we have chosen  $\varepsilon = \dfrac{1}{2k},\ k\neq 0$. 
Therefore, the invariant surface conditions are:
\begin{eqnarray}\label{torred1}
\begin{cases}
    k\psi+(1+2kt)\frac{\partial \psi}{\partial t} &= 0,\\
    k\varphi-(1+2kt)\frac{\partial\varphi}{\partial t} &= 0.
\end{cases}
\end{eqnarray}
Solving the system~\eqref{torred1} by the  method of characteristics we get
\begin{eqnarray*}
\begin{aligned}
    \psi &=& \frac{1}{\sqrt{1+2kt}}\,F(x),\\
    \varphi &=& \sqrt{1+2kt}\,G(x),
\end{aligned}
\end{eqnarray*}
where $F$ and $G$ are  smooth positive functions of $x$. Substituting these expressions of $\psi$ and $\phi$ to the system~\eqref{sistricciconfor}, we obtain
\begin{eqnarray}\label{siste1tor}
\begin{aligned}
-\frac{k}{F^2}\delta_{ij}=&(n-2)\frac{F_{ij}}{F}+\left(\frac{\Delta F}{F}-(n-1)\frac{|\nabla F|^2}{F^2}\right)\delta_{ij}-\frac{m}{G}\Big(G_{ij}+\frac{F_iG_j+F_jG_i}{F}\\
&-\frac{\langle \nabla F,\nabla G\rangle}{F}\delta_{ij}\Big),\\
-kG^2=&\mu-GF^2\Big(\Delta G-(n-2)\frac{\langle \nabla G,\nabla F\rangle}{F}\Big)-(m-1)F^2|\nabla G|^2,
\end{aligned}
\end{eqnarray}
for all $i,j\in\{1,\dots,n\}$. In view of the difficulty in solving the previous system, we are going to further reduce it by considering the case when the base has dimension $1$ and the fiber is the canonical sphere, that is the product manifold $M^1\times S^m$. By this reduction, the system~\eqref{siste1tor} becomes
\begin{eqnarray}\label{simplsis1}
\begin{aligned}
\frac{k}{F^2}&=\frac{m}{G}\left(G_{xx}+\frac{F_xG_x}{F}\right),\\
kG^2&=-\mu+\frac{k}{m}G^2+(m-1)F^2G_x^2,
\end{aligned}
\end{eqnarray}
where $\mu=m-1$
. Therefore, from the second one we have
\begin{equation*}
\left(FG_x\right)^2=\frac{k}{m}G^2 +1\implies G_x=\pm\frac{1}{F} \sqrt{\frac{k}{m}G^2 +1}.
\end{equation*}
Integrating it we get
\begin{equation*}
G(x)=\begin{cases}\pm\sqrt{\frac{m}{k}} \sinh \left(\sqrt{\frac{k}{m}}  \int_{x_0}^x \frac{d \tau}{F(\tau)} \right),& k>0,\\ \pm\sqrt{\frac{m}{-k}} \sin \left(\sqrt{\frac{-k}{m}} \int_{x_0}^x \frac{d \tau}{F(\tau)}\right),& k<0.\end{cases}
\end{equation*}
Observe that this solution satisfies identically the first equation of the system~\eqref{simplsis1}. Furthermore, the solution suggests the following parameterization:
\begin{equation*}
s(x)=\int_{x_0}^x \frac{d \tau}{F(\tau)}.
\end{equation*}
Therefore
\begin{equation*}
\frac{(1+2 k t)}{F^2(x)}  d x \otimes d x=(1+2 k t) d s \otimes d s .
\end{equation*}
So, having all the solutions of system~\eqref{simplsis1} at hand,  when $k\ne0$, we arrive at the following solutions of the Ricci flow for the warped space $M^1\times S^n$: 
\begin{eqnarray*}
\begin{aligned}
    g&=&(1+2k^2t)\left(ds\otimes ds+\frac{m}{k^2}
    \sinh^2\left(\sqrt{\frac{k^2}{m}} s\right) g_{can}\right),\\
    g&=&(1-2k^2t)\left(ds\otimes ds+\frac{m}{k^2}\sin^2\left(\sqrt{\frac{k^2}{m}} s\right) g_{can}\right),
\end{aligned}
\end{eqnarray*}
when $k\neq 0$. Note that the first metric yields the standard hyperbolic $m$-space, see~\cite{ma2011ricci}. Both of them are well known solutions in the bibliography found independently one from the other. This is an example of how symmetries can unify solutions found with different \emph{ad hoc} methods.

\subsection{Invariant Solutions for Doubly-Warped Product Manifolds}
Firstly, we need to express explicitly the Ricci flow equations. In this case, see~\cite{petersen2006riemannian}
\begin{eqnarray}\label{eq:DW}
\begin{aligned}
    \chi\chi_t&=\frac{p}{\varphi}\left(\varphi_{xx}-\frac{\varphi_x\chi_x}{\chi}\right)+\frac{q}{\psi}\left(\psi_{xx}-\frac{\psi_{x}\chi_{x}}{\chi}\right),\\
    \varphi\varphi_{t}&=-(p-1)+\frac{\varphi^2}{\chi^2}\left[\frac{1}{\varphi}\left(\varphi_{xx}-\frac{\varphi_x\chi_x}{\chi}\right)+(p-1)\left(\frac{\varphi_x}{\varphi}\right)^2+q\left(\frac{\varphi_x\psi_x}{\varphi\psi}\right)\right],\\
    \psi\psi_t&=-(q-1)+\frac{\psi^2}{\chi^2}\left[\frac{1}{\psi}\left(\psi_{xx}-\frac{\psi_x\chi_x}{\chi}\right)+(q-1)\left(\frac{\psi_x}{\psi}\right)^2+p\left(\frac{\varphi_x\psi_x}{\psi\varphi}\right)\right].
\end{aligned}
\end{eqnarray}
By the same argument as in the previous subsection, we will make use of the symmetry
\begin{equation*}
  X=(1+2kt)\frac{\partial}{\partial t}+k\chi \frac{\partial}{\partial \chi}+k\varphi\frac{\partial}{\partial \varphi}+k\psi\frac{\partial}{\partial \psi},
\end{equation*}
where $k\neq 0$. Therefore, the corresponding invariant surface conditions are
\begin{eqnarray}\label{redo}
\begin{aligned}
    Q_\chi=k\chi-(1+2tk)\chi_t&=0,\\
    Q_\varphi=k\varphi-(1+2kt)\varphi_t&=0,\\
    Q_\psi=k\psi-(1+2kt)\psi_t&=0.
\end{aligned}
\end{eqnarray}
Solving  system~\eqref{redo} we get
\begin{eqnarray*}
\begin{aligned}
    \chi(x,t)&=\sqrt{1+2tk}\,F(x),\\
    \varphi(x,t)&=\sqrt{1+2tk}\,G(x),\\
    \psi(x,t)&=\sqrt{1+2tk}\,H(x).\\
\end{aligned}
\end{eqnarray*}
Inserting these expressions of $\chi$, $\psi$ and $\varphi$ to the system~\eqref{eq:DW}, we arrive to the reduced system
\begin{eqnarray}\label{red02}
\begin{aligned}
    kF^2&=\frac{p}{G}\left(G''-\frac{G'F'}{F}\right)+\frac{q}{H}\left(H''-\frac{H'F'}{F}\right),\\
    kG^2&=-(p-1)+\left(\frac{G}{F}\right)^2\left(\frac{G''}{G}-\frac{G'F'}{FG}+(p-1)\left(\frac{G'}{G}\right)^2+q\left(\frac{G'H'}{GH}\right)\right),\\
    kH^2&=-(q-1)+\left(\frac{H}{F}\right)^2\left(\frac{H''}{H}-\frac{H'F'}{FH}+(q-1)\left(\frac{H'}{H}\right)^2+p\left(\frac{G'H'}{GH}\right)\right).
\end{aligned}
\end{eqnarray}
Taking inspiration from the procedure for obtaining the solutions for the case of the single warped product manifolds, we start by parametrising $G$ and $H$ by the arc length,
\begin{equation*}
   s(x)=\int_{x_0}^{x}F(\tau)d\tau.
\end{equation*}
The system~\eqref{red02} becomes
\begin{eqnarray}\label{red03}
\begin{aligned}
    k&=p\frac{G''(s)}{G(s)}+q\frac{H''(s)}{H(s)},\\
    k&=\frac{G''(s)}{G(s)}+(p-1)\left(\frac{G'(s)^2-1}{G(s)^2}\right)+q\frac{G'(s)H'(s)}{G(s)H(s)},\\
    k&=\frac{H''(s)}{H(s)}+(q-1)\left(\frac{H'(s)^2-1}{H(s)^2}\right)+p\frac{G'(s)H'(s)}{G(s)H(s)}.
\end{aligned}
\end{eqnarray}
Once more, by looking at the solutions found in the previous section, we start looking for particular solutions of the same form. Our investigation was fruitful, the following functions
\begin{eqnarray*}
\begin{aligned}
    G&=\pm\sqrt{\frac{p+q}{-k}}\sin\left(\sqrt{\frac{-k}{p+q}}s\right)&\text{and } H&=\pm\sqrt{\frac{p+q}{-k}}\cos\left(\sqrt{\frac{-k}{p+q}}s\right),\\
   G&=\pm\sqrt{\frac{(p-1)(p+q)}{-k(p+q-1)}}\sin\left(\sqrt{\frac{-k}{p+q}}s\right) &\text{and } H&=\pm\sqrt{\frac{(q-1)(p+q)}{-k(p+q-1)}}\sin\left(\sqrt{\frac{-k}{p+q}}s\right), 
\end{aligned}
\end{eqnarray*}
for $k<0$, while
\begin{eqnarray*}
\begin{aligned}
   G&=\pm\sqrt{\frac{(p-1)(p+q)}{k(p+q-1)}}\sinh\left(\sqrt{\frac{k}{p+q}}s\right) &\text{ and } H&=\pm\sqrt{\frac{(q-1)(p+q)}{k(p+q-1)}}\sinh\left(\sqrt{\frac{k}{p+q}}s\right), 
\end{aligned}
\end{eqnarray*}
for $k>0$, are indeed solutions of system~\eqref{red03}.

Hence, our study yields the following special solutions: 
\begin{eqnarray*}
\begin{aligned}
 g(s,t)=&(1-2k^2t)ds\otimes ds+(1-2k^2t)\frac{p+q}{k^2}\sin^2\left(\sqrt{\frac{k^2}{p+q}}s\right)g_{S^p _{can}}\\
    &+(1-2k^2t)\frac{p+q}{k^2}\cos^2\left(\sqrt{\frac{k^2}{p+q}}s\right)g_{S^q_{can}},\\
    g(s,t)=&(1-2k^2t)ds\otimes ds+(1-2k^2t)\frac{(p-1)(p+q)}{k^2(p+q-1)}\sin^2\left(\sqrt{\frac{k^2}{p+q}}s\right)g_{S^p _{can}}\\
    &+(1-2k^2t)\frac{(q-1)(p+q)}{k^2(p+q-1)}\sin^2\left(\sqrt{\frac{k^2}{p+q}}s\right)g_{S^q_{can}},\\
     g(s,t)=&(1+2k^2t)ds\otimes ds+(1+2k^2t)\frac{(p-1)(p+q)}{k^2(p+q-1)}\sinh^2\left(\sqrt{\frac{k^2}{p+q}}s\right)g_{S^p _{can}}\\
    &+(1+2k^2t)\frac{(q-1)(p+q)}{k^2(p+q-1)}\sinh^2\left(\sqrt{\frac{k^2}{p+q}}s\right)g_{S^q_{can}},
\end{aligned}
\end{eqnarray*}
where $k\ne0$. To the best of our knowledge, these solutions are new in the literature.

\section{Conclusion and Discussion}
In the present work, we determine the Lie point symmetries of the Ricci flow in arbitrary dimensions. By using their algebraic properties we are able to ``recycle'' them in order to expeditiously obtain the Lie point symmetries of the Ricci flow for a particular family of metrics. A task that can be very challenging if one starts each time from scratch, bearing in mind that even a four-dimensional metric yields a system of determining equations that involves almost two million PDE!

Unfortunately, classical symmetries, being a very broad notion that can be applied virtually to any kind of system of differential equations, have a very serious drawback; they usually will  not help us to obtain the kinds of solutions that we are looking for. And this is apparent in our case: The infinite dimensional sub algebra merely points to the tensorial nature of our system, while the finite one says, on one hand, that the system is autonomous and on the other hand usher us to consider separable solutions.  Solutions that are known to hold if and only if the initial manifold is an Einstein one. A fact that limits our choices regarding the inicial condition to be considered. 

One possible way out is by considering more ``exotic'' kinds of symmetries, like the non classical ones, see Section $9.3$ in \cite{hydon2000symmetry}. In a word, non classical symmetries are symmetries that are admitted only by specific families of solutions of a differential equation and not by the differential equation itself, as is the case with the classical symmetries.  Although our initial studies were promising none of the non classical symmetries found yielded non trivial solutions. This do not come as  a surprise since the systems of differential equations involved are now non linear. Therefore, a more exhausting classification must be carried over in order to obtain non classical symmetries that can yield solutions of the Ricci flow problem of some interest.

\section*{Acknowledgements}
Enrique F. L. Agila has been partially supported by Coordenação de Aperfeiçoamento de Pessoal de Nível Superior (CAPES),  Grant 001.

\end{document}